\documentclass[11pt,leqno]{amsart}

\usepackage{bm, latexsym, amsmath, amsthm, amsopn, amsfonts,graphicx, amssymb,epsfig,graphics}
\usepackage{yhmath}
\usepackage{graphics,epsfig} 
\usepackage{color, xcolor}
\usepackage{tikz}
\usetikzlibrary{patterns}

\usepackage{hyperref}
\usepackage{color, xcolor}
\definecolor{nota}{rgb}{1,0,0.8}

\newtheorem{thm}{Theorem}[section]

\newtheorem{cor}[thm]{Corollary}
\newtheorem{lemma}[thm]{Lemma}
\newtheorem{prop}[thm]{Proposition}
\newtheorem{defn}[thm]{Definition}
\newtheorem{rem}[thm]{Remark}

\newtheorem{conj}{Conjecture}

\def\N{\mathbb{N}}
\def\R{\mathbb{R}}
\newcommand{\bd}{\partial}

\newcommand{\ep}{\varepsilon}

\def\norma#1#2{\|#1\|_{\lower 4pt \hbox{$\scriptstyle #2$}}}

\newcommand\la{\lambda}

\begin{document}
	
\title{On the quantitative isoperimetric inequality in the plane}
		
\author{Chiara Bianchini}
\author{Gisella Croce} 
\author{Antoine Henrot} 

\address{\emph{C. Bianchini:}  Dipartimento di Matematica ed Informatica ``U.~Dini'',
	Universit\`a di Firenze, Viale Morgagni 67/A, 50134 Firenze, Italy.}
	\email{chiara.bianchini@math.unifi.it}

\address{\emph{G. Croce:} 
Normandie Univ, France; ULH, LMAH, F-76600 Le Havre; FR CNRS 3335, 25 rue Philippe Lebon, 76600 Le Havre, France.}
\email{gisella.croce@univ-lehavre.fr}

\address{\emph{A. Henrot:} Institut \'Elie Cartan de Lorraine UMR CNRS 7502, Universit\'e de Lorraine, BP 70239 54506 Vandoeuvre-les-Nancy Cedex, France.} 
\email{antoine.henrot@univ-lorraine.fr}



\maketitle

\begin{abstract}
In this paper we study the quantitative isoperimetric inequality in the plane. We prove the existence of a 
set $\Omega$, different from a ball,  which minimizes the ratio $\delta(\Omega)/\lambda^2(\Omega)$, where 
$\delta$ is the isoperimetric deficit and $\lambda$ the Fraenkel asymmetry, giving a new proof of
the quantitative isoperimetric inequality. Some new properties of the optimal set are also shown.
\end{abstract}


\section{Introduction}
The last few years have seen several remarkable breakthroughs in the study of quantitative isoperimetric 
inequalities. These are refinements of the classical isoperimetric inequality, since they control the area distance to the ball with a function of the perimeters difference, and can be viewed as stability results. In this paper we will deal with the so called \emph{Fraenkel asymmetry} and the \emph{isoperimetric deficit}. 

Let $\Omega\subset \R^N$ be a  Borel set, with Lebesgue measure $|\Omega|$. Its isoperimetric deficit is defined as
$$
\delta(\Omega)=\frac{P(\Omega)-P(B)}{P(B)}, \quad |B|=|\Omega|,
$$
where $B$ is a ball and $P(\Omega)$ is the perimeter of $\Omega$ in the sense of De Giorgi.
 The isoperimetric inequality guarantees that $\delta(\Omega)$ is positive and null only if $\Omega$ is a ball.
The Fraenkel asymmetry of the set $\Omega$
is defined as
$$
\lambda(\Omega)=\min_{x \in \R^N}\left\{\frac{|\Omega\Delta B_x|}{|\Omega|}, \ |B_x|=|\Omega|\right\},
$$
where $B_x$ is a ball centered at $x$.
It is a natural $L^1$ distance between $\Omega$ and its {\it closest} ball.
The quantitative isoperimetric inequality, first proved by N. Fusco, F. Maggi and A. Pratelli in \cite{FMP}, affirms the existence of a constant $C_N$, depending only on the dimension, such that for every $\Omega \subset \R^N$ one has
\begin{equation}\label{qisop}
\lambda^2(\Omega) \leq C_N \delta(\Omega).
\end{equation}
See \cite{Ci-Le}, \cite{FiMP} and \cite{LZDJ} for alternative proofs of the same result. For more on the surrounding literature, in particular for variants or generalizations of inequality (\ref{qisop}), we refer to \cite{Fuscopreprint} and \cite{Maggibulletin} and the references therein.

\medskip
Despite this recent progress, the optimal value $C^*_N$ of the constant
in (\ref{qisop}) is still not known, even in dimension two. The value of $C^*_N$ can be defined as
\begin{equation}\label{Shop}
\displaystyle \frac{1}{C_N^*}\,=\inf \left\{\frac{\delta(\Omega)}{\lambda^2(\Omega)}\;,\Omega\subset\R^N, \Omega\not= B\right\}\,.
\end{equation} 
This can be seen as a non-standard shape optimization problem,  since the class of admissible domains
is the class of all measurable sets but the balls.
In this paper we prove existence of a minimizer for this problem in the plane:
\begin{thm}\label{thm_existence_minimizer}
There exists  a set $\Omega_0$
which minimizes the shape functional
$$
\mathcal{F}(\Omega)=\frac{\delta(\Omega)}{\lambda^2(\Omega)},
$$ 
among all the subsets of $\R^2$ 
(the ball excluded).
\end{thm}
This existence result in fact provides a new proof of the quantitative isoperimetric inequality in the plane.

\medskip
In \cite{HH}, \cite{HHW}, \cite{Ca}, \cite{AFN} (see also \cite{AFN2} and \cite{Ca2}), the minimization problem (\ref{Shop}) was studied in the restricted class of planar convex sets. In particular, one may find the following result in \cite{AFN}. 
\begin{thm}[\cite{AFN}]\label{thmAFN}
Let $\mathcal{C}$ be the class of planar convex sets; then
$$
\inf_{\Omega\in\mathcal{C}} \mathcal{F}(\Omega)=0.405585\,,
$$
and the minimum is attained at a particular ``stadium''.
\end{thm}
This optimal stadium will be useful for us in excluding possible minimizing sequences
converging to the ball.
In  \cite{CiLe2}, Cicalese and Leonardi addressed the same question, among all the subsets of $\R^N$,  by considering
the functional $\widehat{\mathcal{F}}(\Omega)$ (extended by relaxation to the ball $B$) defined by:
$$\widehat{\mathcal{F}}(\Omega)=\left\lbrace\begin{array}{lc}
\mathcal{F}(\Omega) & \mbox{if $\Omega\not= B$},\\
\inf \{\liminf \mathcal{F}(\Omega_n), \lambda(\Omega_n)>0, |\Omega_n\Delta B|\to 0\} & \mbox{if $\Omega=B$}.
\end{array}\right.
$$
By using an \emph{iterative selection principle} and by 
applying Bonnesen's annular symmetrization, they showed that a minimizing sequence for the above infimum is 
made up of \emph{ovals}, that is, $C^1$ convex sets, with two 
orthogonal axes of symmetry, whose boundary is the union of two congruent arcs of circle. 
They proved this result using properties of this family of sets established in \cite{Ca} and \cite{AFN}.
  
In this paper we present a different approach based on a new kind of symmetrization. We replace any set $\Omega$ by a new set $\Omega^*$ having two orthogonal axes of symmetry and whose boundary is composed by four arcs of circle. This is done in such a way that the areas of $B\setminus \Omega$ and $\Omega\setminus B$, where $B$ is an optimal ball, are conserved. 
The key point in our construction is Proposition \ref{prop_simmetrizzazione} where we prove that, for sets converging to the ball, this symmetrization decreases the functional $\mathcal{F}$ asymptotically.
It suffices then to explicitly compute  $\mathcal{F}$ for this family of sets and to prove that the limit of the 
symmetrized minimizing sequence  is greater than the quantity $\frac{\pi}{8(4-\pi)}$. 
This is done in Theorem \ref{casABCD}, Section \ref{sec2}.
Finally, Theorem \ref{thmAFN} shows that any sequence converging to the ball is excluded from being a minimizing sequence for $\mathcal{F}$, thereby concluding the argument.

In Section \ref{sec3}, we prove the existence result, Theorem \ref{thm_existence_minimizer}. 
Once we know that a minimizing sequence $\{\Omega_n\}$ cannot converge to the ball, we need to prove that it is
uniformly bounded in a sufficiently large set $R$.
This is done by proving that we can replace the sequence $\{\Omega_n\}$ by a new sequence
with a fixed finite number of connected components, see  Proposition \ref{perimeter}. We conclude
in a classical way, using the compact embedding $BV(R)\hookrightarrow L^1(R)$.

In Section \ref{sec4} we establish some qualitative properties of the optimal set. 
Our main contribution is the proof that the optimal set has at least two balls realizing the Fraenkel asymmetry. The case by case proof we provide is lengthy, but rather simple. We also give a bound on the number of connected components of an optimal set in Theorem \ref{thmqualit}.

In the last section, we present in detail a possible (non convex) candidate for the optimal set for which the value of the functional $\mathcal{F}$ approximatively equals 0.3931. This would give a value of 2.5436 for the optimal constant $C_2^*$. 
Notice that the same value $C^*_2$ was conjectured in \cite{CiLe2}, where it was also shown that it corresponds to  a certain shape referred to as a \emph{mask}. We will keep this denomination. We also explain what remains to show that the mask is an optimal set for $\mathcal{F}$.


\section{Sequences converging to a ball and a new rearrangement}\label{sec2}
Since the class of admissible domains for $\mathcal{F}$ is composed of any domains but the balls, a convergent minimizing sequence
has two possible behaviours:
 either it converges to a ball
 or it converges to a domain $\Omega_0$ different to a ball and this will provide
a minimizer of $\mathcal{F}.$
Notice that, since we are dealing with the Fraenkel asymmetry, the convergences of sets that we consider are convergences of the measure of the symmetric difference to zero.

In this section we are going to study the behaviour of sequences of sets converging to a ball. 
To do that, we will define a new kind of symmetrization of sets which turns out
to decrease (at least asymptotically) the functional $\mathcal{F}$. Since our functional is scale invariant,
without loss of generality, we can fix the area of admissible domains equal to $\pi$, if not differently stated.

In the sequel a special class of sets, named \emph{transversal}, will be useful for our pourposes.
\begin{defn}\label{transversal}
	We say that a set $\Omega$ is transversal to a ball $B$ if the number of intersection points $M_i\equiv(x_i,y_i), i\in \{1,...,2p\}$, between  $\partial \Omega$ and $\partial B$ is finite.
\end{defn}

\begin{defn}\label{rearrangement}
Let $\Omega$ be a planar set and let $B$ be a ball such that $|\Omega|=|B|$ and $\Omega$ transversal to $B$. We define the rearranged set $\Omega^*$ as follows.
Let
$$
\Omega^{OUT}=\Omega \setminus B
\,,\,\,\,\,\,\,\,\,\Omega^{IN}= B \setminus \Omega\,.
$$
Let us now consider the parts of the circle which bound $\Omega^{IN}$ and $\Omega^{OUT}$.
Let $\gamma^{IN}=\mathcal{H}^1(\partial \Omega^{IN}\setminus \partial \Omega)$
and $\gamma^{OUT}=\mathcal{H}^1(\partial \Omega^{OUT}\setminus \partial \Omega)=\mathcal{H}^1(\partial B)-\gamma^{IN}$.

Let $A_1, A_2, A_3, A_4$ be four points on  $\partial B$  defined as follows. The length of the arcs of circle   $\sigma_{12}$ and $\sigma_{34}$, with endpoints
 $A_1, A_2$, and $A_3, A_4$, respectively, equals $\gamma^{OUT}/2$.
 The length of the arcs of circle  $\sigma_{23}$ and $\sigma_{41}$, with endpoints
 $A_2, A_3$, and $A_4, A_1$, respectively, equals $\gamma^{IN}/2$.
We next  consider another arc of circle, $\tilde{\sigma}_{12}$, with endpoints $A_1, A_2$, 
outside $B$, such that the measure of the surface $a_{12}$ between
$\tilde{\sigma}_{12}$ and
$\sigma_{12}$ is equal to $|\Omega^{OUT}|/2$.
In an analogous way we define $\tilde{\sigma}_{34}$.
We consider an arc of circle $\tilde{\sigma}_{23}$ with endpoints
 $A_2, A_3$,
inside $B$, such that the measure of the surface $a_{23}$ between
$\tilde{\sigma}_{23}$ and 
$\sigma_{23}$ is equal to  $|\Omega^{IN}|/2$.
We define $\tilde{\sigma}_{41}$ in an analogous way.

$\Omega^*$ is the set whose boundary is the union of the arcs $\tilde{\sigma}_{12}, \tilde{\sigma}_{23}, \tilde{\sigma}_{34}, \tilde{\sigma}_{41}$.
\end{defn}

\begin{rem}
The previous rearrangement can be extended to non-transversal sets in a natural way. In this general case, the boundary
of $\Omega^*$ will contain 4 congruent arcs of the boundary of the ball with length $(2\pi-\gamma^{IN}-\gamma^{OUT})/4$ each.
We point out that we will use this rearrangement
for domains $\Omega$
such that $|\Omega \Delta B|$ will be small enough, in such a way that the above construction is always possible.
\end{rem}

\newcommand{\omegain}{(-50:5)arc(-30:30:7.6)} 
\newcommand{\omegainn}{(130:5)arc(150:210:7.6)}
\newcommand{\disegnoomega}{(0:5.5)to[out=90,in=-30](30:4.3)to[out=150,in=0](90:6.5)to[out=180,in=45](135:4.3)to[out=225,in=90](180:5.5)to[out=270,in=210](240:4.6)to[out=30,in=150](270:4.5)to[out=-30,in=210](320:4.7)to[out=30,in=270](0:5.5)}
\newcommand{\pillow}{
	\fill[opacity=0.7, gray] (50:5)arc(20:160:3.42);\fill[opacity=0.7,gray](230:5)arc(200:340:3.42);
	\fill[white](0,0)circle(5);
	\fill[gray!20] (-50:5)arc(-50:50:5); \fill[gray!20](130:5)arc(130:230:5);
	\begin{scope}
		\clip (0,0)circle(5);
		\fill[white] \omegain;
	\end{scope}
	\begin{scope}
		\clip (0,0)circle(5);
		\fill[white] \omegainn;
	\end{scope}
	\draw (0,0) circle(5); 
	\draw[blue] (120:6.5)node{$\Omega^*$};
	\draw[thick, blue] (50:5)arc(20:160:3.42);
	\draw[thick, blue] (-50:5)arc(-30:30:7.6); 
	\draw[thick,blue](130:5)arc(150:210:7.6); 
	\draw[thick, blue] (230:5)arc(200:340:3.42);
	\fill (50:5)node[right]{$A_2$}circle(2pt);\fill(130:5)node[left]{$A_1$}circle(2pt);\fill (230:5)node[left]{$A_4$}circle(2pt);\fill (-50:5)node[right]{$A_3$}circle(2pt);
	\draw (170:6)node{$\sigma_{41}$};\draw (260:4)node{$\sigma_{34}$};\draw (10:6)node{$\sigma_{23}$};\draw (100:4)node{$\sigma_{12}$};
	\draw[blue] (190:3)node{$\widetilde{\sigma}_{41}$};\draw[blue] (285:7)node{$\widetilde{\sigma}_{34}$};\draw[blue] (-10:3)node{$\widetilde{\sigma}_{23}$};\draw[blue] (80:7)node{$\widetilde{\sigma}_{12}$};
	\begin{scope}[xshift=-6cm]
		\fill[opacity=0.7, gray] \disegnoomega;
		\fill[gray!20](0,0)circle(5);
		\begin{scope}
			\clip \disegnoomega;
			\fill[white] (0,0)circle(5);
		\end{scope}
		\draw (0,0)circle(5);
		\draw[thick, blue] \disegnoomega;
		\draw[blue] (75:6)node[right]{$\Omega$};
		\draw [<-] (5:5.2)--(350:6.5)node[right]{$\Omega^{OUT}$};
		\draw[<-] (140:4.7)--(140:5.5)node[left]{$\Omega^{IN}$};
		\draw[thick,->] (35:8)arc(120:60:7); 
	\end{scope}
}
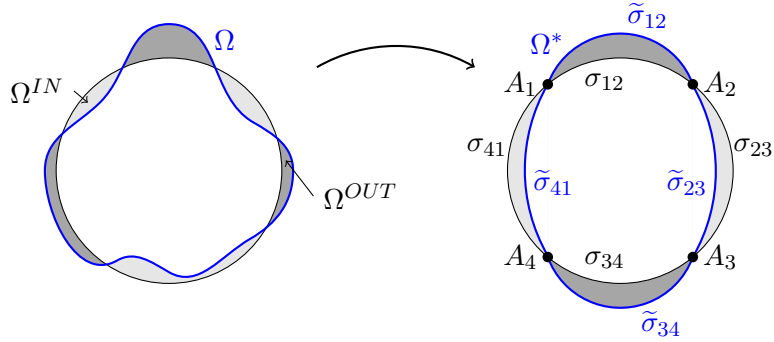
\begin{figure}[h]
	\centering
	\begin{tikzpicture}[x=3mm,y=3mm]
	\pillow
	\end{tikzpicture}
	\caption{A set $\Omega$ and its symmetrization $\Omega^*$.}\label{pillow}
\end{figure}

\begin{rem}
The previous symmetrization does not coincide with the circular Bonnesen symmetrization, used in \cite{Ca2} and \cite{CiLe2}. 
This can be easily seen by noticing that the circular Bonnesen symmetrization of a set $\Omega$ necessarily intersects the largest ball containing $\Omega$ (at least at one point). This is not the case for the symmetrized set $\Omega^*$.
Moreover the boundary of the set $\Omega^*$ is in general not  of class $C^1$, even in the convex case.
\end{rem}

The first order optimality condition satisfied by a ball which realizes the Fraenkel asymmetry gives 
 a constraint on the coordinates of the intersections points. 
More precisely we prove the following.
\begin{prop}\label{cdt_optimalite}
	Let  $\Omega$ be a transversal set to an optimal ball $B$. Then the intersection points $M_i\equiv(x_i,y_i), i\in\{1,..., 2p\}$ between $\partial\Omega$ and 
	$\partial B$ satisfy
	\begin{eqnarray*}
	x_1+x_3+...+x_{2p-1} - (x_2+x_4+...+ x_{2p})=0,\\
	y_1+y_3+...+y_{2p-1} - (y_2+y_4+...+ y_{2p})=0.
	\end{eqnarray*}
\end{prop}

\begin{rem}\label{remarknumberintersection}
The assumption that $\Omega$ is transversal to the ball $B$ will be used to prove that the function $\psi(x,y)=|B_{(x,y)}\Delta \Omega|$ is differentiable.
	Moreover, if $\Omega$ is  a transversal set to an optimal ball $B$, then by previous result, the intersection points between the boundary of a minimizing set $\partial \Omega$ and $\partial B$ are at least four.
\end{rem}

Proposition \ref{cdt_optimalite} is in fact a corollary of the following differentiability result.
\begin{lemma}\label{lemmapsi}
Assume that $\Omega$ is transversal to a ball $B_{(x,y)}$ centered in ${(x,y)}$.  Then 
the function $\psi(x,y)=|\Omega\Delta B_{(x,y)}|$ is differentiable and 
\begin{equation}\label{DpsiDy}
\frac{\partial \psi}{\partial y}=-2\Big(x_1+x_3+...+x_{2p-1} - (x_2+x_4+...+ x_{2p})\Big).
\end{equation}
\end{lemma}
\begin{proof}
Notice that 
$$
|\Omega\Delta B_{(x,y)}|=\int_{\R^2}|\chi_{\Omega}-\chi_{B_{(x,y)}}|^2\;dxdy=\int_{\R^2}\chi_{\Omega}\;dxdy-2\int_{\R^2}\chi_{\Omega}\chi_{B_{(x,y)}}\;dxdy+\int_{\R^2}\chi_{B_{(x,y)}}\;dxdy,
$$
that is, $\displaystyle \psi({(x,y)})=2\pi-2\int_{B_{(x,y)}}\chi_{\Omega}$.
Moreover
$$
\int_{B_{{(x,y)}}}\chi_{\Omega}=\int_{x-1}^{x+1}\int_{y-\sqrt{1-(u-x)^2}}^{y+\sqrt{1-(u-x)^2}}\chi_{\Omega}(u,v)\;dv\;du.
$$
Now,  the number of points where $\displaystyle \int_{B_{{(x,y)}}}\chi_{\Omega}$ is not differentiable is finite, due to the assumption that $\Omega$ is transversal to $B_{(x,y)}$. Therefore one can compute the derivative:
$$
\frac{\bd}{\bd y}\int_{B_{(x,y)}}\chi_\Omega  =\int_{x-1}^{x+1}\left( \chi_{\Omega}(u,y+\sqrt{1-(u-x)^2}) - \chi_{\Omega}(u,y-\sqrt{1-(u-x)^2})\right)\; du.
$$
Notice that these two integrals measure the length of the projection on the horizontal axis of $\partial B \cap \Omega$
(counted positively on the upper half plane, negatively on the lower half plane).
This entails (see Figure \ref{fig-epi})
$$
\frac{\bd}{\bd y}\int_{B_{(x,y)}}\chi_\Omega =(x_1-x_{2})+(x_3-x_4)+...+(x_{2k+1}-x_{2k})+(x_{2p-1}-x_{2p})].
$$
\end{proof}

In the following theorem, we study the asymptotic behaviour  of $\mathcal{F}(\Omega_{\varepsilon} ^*)$,  where $\Omega_{\varepsilon}$ is a sequence
of sets converging to a ball. In particular, we prove that the limit value is always greater than the value
of $\mathcal{F}$ for the optimal stadium of Theorem \ref{thmAFN}. We will work in the general case where the
boundary of $\Omega_{\varepsilon}^*$ may contain arcs of the ball (see Figure \ref{thetaeta}), even if we will use later this
theorem only for transversal domains.
\begin{thm}\label{casABCD}
Let $\{\Omega_{\varepsilon}\}_{ \varepsilon>0}$,  be a sequence of sets, such that $|\Omega_{\varepsilon}|=\pi=|B|$ where $B$ is a ball. 
 Assume that  
$|B\Delta \Omega_{\varepsilon} |= {4\varepsilon}$.
Then
$$
\liminf_{\varepsilon\to 0}\mathcal{F}({\Omega_{\varepsilon} ^*})\geq \dfrac{\pi}{8(4-\pi)}\approx0.45.
$$
\end{thm}
From now on we will use the  following functions:
\begin{equation}\label{defn_fonctions_gh}
g(t)=t-\sin(t)\cos(t),
\qquad h(t)=\frac {g(t)}{\sin^2(t)}\,.
\end{equation}
\begin{proof}
Let us fix some notations for $\Omega_{\varepsilon}^*$; let us consider the case (a) of Figure \ref{thetaeta}.
According to the figure,  $R_1^{\varepsilon} =\frac {\sin(\eta_1^{\varepsilon})}{\sin(\theta_1^{\varepsilon})}$ and
$R_2^{\varepsilon}=\frac {\sin(\eta_2^{\varepsilon})}{\sin(\theta_2^{\varepsilon})}$ are the radii of the arcs $A_1B_1,A_2B_2$, respectively. 
We observe that 
$$
0\leq \eta_1^{\varepsilon}+\eta_2^{\varepsilon}\leq \pi/2, \qquad \eta_1^{\varepsilon}\leq \theta_1^{\varepsilon}\leq \pi, \qquad \theta_2^{\varepsilon}\leq \eta_2^{\varepsilon}\leq  \pi/2.
$$
\newcommand{\casoA}{(30:5)arc(15:-15:9.66);}
\newcommand{\casoAA}{(150:5)arc(165:195:9.66);}
\newcommand{\figuraThetaEta}{
\fill[blue!10](70:5)arc(45:135:2.42);\fill[blue!10](250:5)arc(225:315:2.42);\fill[white](0,0)circle(5);
\fill[blue!10] (-30:5)arc(-30:30:5);
\fill[blue!10] (150:5)arc(150:210:5);
	\begin{scope}
		\clip (0,0)circle(5);
		\fill[white] \casoA
	\end{scope}
	\begin{scope}
		\clip (0,0)circle(5);
		\fill[white] \casoAA;
	\end{scope}
\fill[gray!20] (0,3)--(0,4)arc(90:135:1);\draw (0,2.7)++(110:1.5)node{\footnotesize$\theta_1$};
\fill[gray!50] (0,0)--(0,1.5)arc(90:110:1.5); \draw (88:1.7)node{\footnotesize{$\eta_1$}};
\fill[gray!20] (-2.5,0)--(-1.5,0)arc(0:20:1); \draw (-1,0)node{\footnotesize{$\theta_2$}};
\fill[gray!50] (0,0)--(1.5,0)arc(0:30:1.5);\draw (10:2)node{\footnotesize$\eta_2$};
	\draw (0,0) circle(5);\draw[dotted] (0,-6)--(0,6);\draw[dotted] (-6,0)--(6,0);
	\draw[very thick, blue] (70:5)arc(45:135:2.42);
	\draw[very thick, blue] (70:5)arc(70:30:5);
	\draw[very thick, blue] (30:5)arc(15:-15:9.66);\draw[very thick,blue](290:5)arc(290:330:5); 
	\draw[very thick, blue] (110:5)arc(110:150:5);
	\draw[very thick,blue] (150:5)arc(165:195:9.66);\draw[very thick,blue](210:5)arc(210:250:5);
	\draw[very thick,blue] (250:5)arc(225:315:2.42);
	\draw[very thin, gray] (0,0)--(70:5);\draw (0,0)node[below]{$O$};\draw (70:5) node[right]{$B_1$};
	\draw[very thin, gray] (0,0)--(110:5);\draw (110:5)node[left]{$A_1$};
	\draw (0,3)--(110:5);\draw (-0.2,3)node[right]{$O_1$};
	\draw[very thin, gray] (0,0)--(30:5);\draw (30:5)node[right]{$A_2$};
	\draw[very thin, gray] (0,0)--(-30:5);\draw (-30:5)node[right]{$B_2$};
	\draw (-2.5,0)--(30:5);\draw (-2.5,0)node[left]{$O_2$};
	\draw[<-](200:5)--(220:7)node[left]{area $\ep$};\draw[->] (220:7)--(265:5.2);
	\draw (-0.2,3.3)node[left]{$R_1$};\draw (-1.5,0.9)node{$R_2$};
	}
\newcommand{\casoB}{(30:5)arc(15:-15:9.66);}
\newcommand{\casoBB}{(150:5)arc(165:195:9.66);}
\newcommand{\figuraThetaEtaB}{
	\fill[blue!10](75:5)arc(30:150:1.5);\fill[blue!10](255:5)arc(210:330:1.5);\fill[white](0,0)circle(5);
	\fill[blue!10] (-20:5)arc(-20:20:5);
	\fill[blue!10] (160:5)arc(160:200:5);\fill[blue!10] (160:5)arc(15:-15:6.61);\fill[blue!10] (20:5)arc(165:195:6.61);
	\fill[gray!20] (0,4)--(0,4.5)arc(90:150:0.5); \draw (0,4.5)node{\footnotesize$\theta_1$};
	\fill[gray!50] (0,0)--(0,1.5)arc(90:105:1.5); \draw (88:1.7)node{\footnotesize{$\eta_1$}};
	\fill[gray!20] (10,0)--(9,0)arc(180:162:1);\draw (9,0.2)node[left]{\footnotesize$\theta_2$};
	\fill[gray!50] (0,0)--(1.5,0)arc(0:20:1.5);\draw (10:2)node{\footnotesize$\eta_2$};
	\draw (0,0) circle(5);\draw[dotted] (0,-6)--(0,6);\draw[dotted] (-6,0)--(10.2,0);
	\draw[very thick, blue] (75:5)arc(30:150:1.5);
	\draw[very thick, blue] (75:5)arc(75:20:5);\draw[very thick,blue](-20:5)arc(340:285:5); \draw[very thick, blue] (200:5)arc(200:255:5); \draw[very thick,blue](105:5)arc(105:160:5);
	\draw[very thick, blue] (20:5)arc(165:195:6.61);
	\draw[very thick,blue] (255:5)arc(210:330:1.5);
	\draw[very thick,blue] (160:5)arc(15:-15:6.61);
	\draw[very thin, gray] (0,0)--(75:5);\draw (0,0)node[below]{$O$};\draw (75:5) node[above right]{$B_1$};
	\draw[very thin, gray] (0,0)--(105:5);\draw (105:5)node[above left]{$A_1$};
	\draw[very thin, gray] (0,4)--(105:5);\draw (0,4)node[below]{$O_1$};
	\draw[very thin, gray] (0,0)--(20:5);\draw (20:5)node[above right]{$A_2$};
	\draw[very thin, gray] (0,0)--(-20:5);\draw (-20:5)node[right]{$B_2$};
	\draw[very thin, gray] (10,0)--(20:5);\draw (10,0)node[right]{$O_2$};
	\draw[<-](185:4.7)--(220:7)node[left]{area $\ep$};\draw[->] (220:7)--(265:5.2);
}	
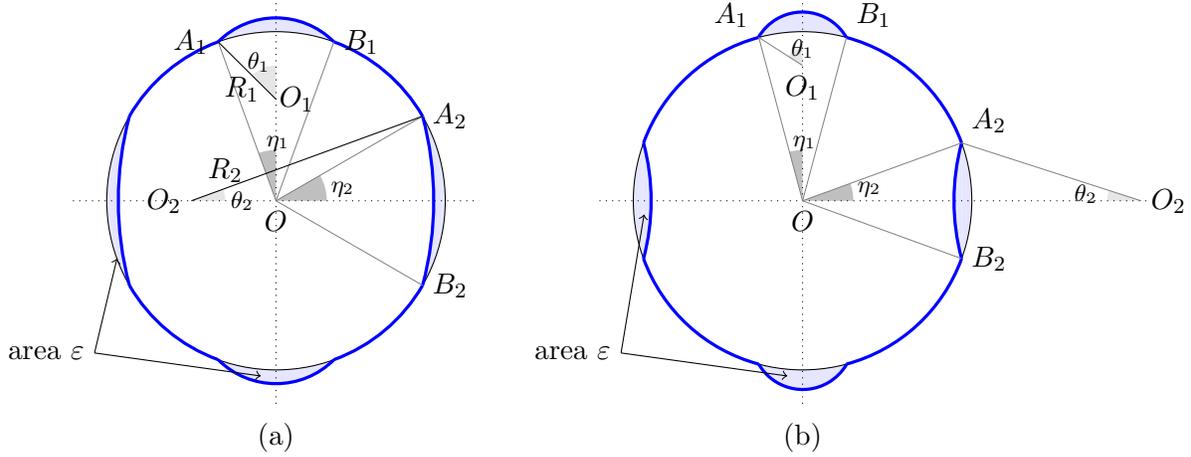
\begin{figure}[h]
	\centering
	\begin{tikzpicture}[x=4.5mm,y=4.5mm]
	\figuraThetaEta
	\draw (0,-7)node{(a)};
	\begin{scope}[xshift=7cm]
	\figuraThetaEtaB
	\draw (0,-7)node{(b)};
	\end{scope}
	\end{tikzpicture}
	\caption{The parametrization of a set $\Omega^*_{\ep}$ in the proof of Theorem \ref{casABCD}.}\label{thetaeta}
\end{figure}

Notice that, by construction, it holds
$$
\lambda(\Omega^*_\ep)=\frac{|\Omega^*_\ep\Delta B|}{\pi}=\frac{4\ep}{\pi},
$$
and hence
\begin{eqnarray}\label{aires}
\varepsilon&=&(R_1^{\varepsilon})^2g(\theta_1^{\varepsilon})-g(\eta_1^{\varepsilon})=\sin^2(\eta_1^{\varepsilon})
h(\theta_1^{\varepsilon})-g(\eta_1^{\varepsilon}),
\\
\varepsilon&=&-(R_2^{\varepsilon})^2g(\theta_2^{\varepsilon})+g(\eta_2^{\varepsilon})=-\sin^2(\eta_2^{\varepsilon})
h(\theta_2^{\varepsilon})+g(\eta_1^{\varepsilon})\,.
\nonumber
\end{eqnarray}
Moreover $\lambda(\Omega_{\varepsilon})\leq \frac{4\varepsilon}{\pi}$ and
\begin{eqnarray*}
\delta(\Omega_{\varepsilon}^*)&=&\frac{1}{2\pi}
\left[
4R_1^{\varepsilon}\theta_1^{\varepsilon}+4R_2^{\varepsilon}\theta_2^{\varepsilon}+4\left(\frac{\pi}{2}-\eta_1^{\varepsilon}-\eta_2^{\varepsilon}\right)-2\pi
\right]\\
&=&
\frac{2}{\pi}
\left(
\sin(\eta_1^{\varepsilon})\frac{\theta_1^{\varepsilon}}{\sin(\theta_1^{\varepsilon})}-\eta_1^{\varepsilon}
+
\sin(\eta_2^{\varepsilon})\frac{\theta_2^{\varepsilon}}{\sin(\theta_2^{\varepsilon})}-\eta_2^{\varepsilon}
\right).
\end{eqnarray*}
We deduce from (\ref{aires}) that
\begin{equation}\label{aires_2}
\frac{\varepsilon}{\sin^2(\eta_1^{\varepsilon})}=h(\theta_1^{\varepsilon})-h(\eta_1^{\varepsilon})\,,\,\,\,\,\,\,
\frac{\varepsilon}{\sin^2(\eta_2^{\varepsilon})}=h(\eta_2^{\varepsilon})-h(\theta_2^{\varepsilon})\,.
\end{equation}
Hence
\begin{equation}\label{deltaA}
\delta(\Omega_\ep^*)=\frac{2}{\pi} \left[F\left(\eta_1^{\varepsilon},\frac{\ep}{\sin^2(\eta_1^{\varepsilon})}\right)+ 
F\left(\eta_2^{\varepsilon},\frac{-\ep}{\sin^2(\eta_2^{\varepsilon})}\right) \right],
\end{equation}
where 
$$
F(x,y)=\sin(x)\frac{h^{-1}(h(x)+y)}{\sin(h^{-1}(h(x)+y))}-x\,.
$$
Observe that $h^{-1}$ exists, since $h'(x)=2\frac{\sin(x)-x\cos(x)}{\sin^3(x)}$ is positive in $(0,\pi)$. 
In the sequel we will omit the dependence of $\eta_i,\theta_i,R_i$ on $\varepsilon$.
Notice that the angles $\eta_1, \theta_1, \eta_2, \theta_2$ may have the following behaviours (up to sub-sequences), 
as $\varepsilon\to 0$: 
\begin{itemize}
	\item [$A_i$:] $\eta_i\to \hat{\eta}_i>0$;
	\item [$B_i$:] $\eta_i\to 0$ and $\frac{\varepsilon}{\sin^2(\eta_i)}\to l_i>0$;
	\item [$C_i$:] $\eta_i\to 0$ and $\frac{\varepsilon}{\sin^2(\eta_i)}\to 0$;
	\item [$D_i$:] $\eta_i\to 0$ and $\frac{\varepsilon}{\sin^2(\eta_i)}\to +\infty$.
\end{itemize}
In each case, we are going to compute the Taylor expansion of $\displaystyle \frac{1}{\varepsilon^2}F\left(\eta_1, \frac{\varepsilon}{\sin^2(\eta_1)}\right)$
and $\displaystyle \frac{1}{\varepsilon^2}F\left(\eta_2, \frac{-\varepsilon}{\sin^2(\eta_2)}\right)$. 
This will help us to estimate from below the limit of $\mathcal{F}(\Omega^*_\varepsilon)$.

Notice that the analysis of case (b) of Figure \ref{thetaeta} is analogous to that  one of case (a).
Indeed we have 
$$
h(\eta_2)+h(\theta_2)=\frac{\ep}{\sin^2(\eta_2)},
$$
which entails
$$
\theta_2=-h^{-1}\Big(h(\eta_2)-\frac{\ep}{\sin^2(\eta_2)}\Big),
$$
since $h$ is an odd function.
Hence in case (b) we obtain the same expression (\ref{deltaA}) for $\delta(\Omega_\ep)$.

\noindent
{\bf Case $A_1$.}
Since $F$ is analytic, we can write 
$$
F(x,y)=\sum_{k,l}\frac{a_{k,l}}{k!l!}(x-\hat{x})^k y^l\,,\,\,\,\,\,\,\,\,a_{k,l}=\frac{\partial^{k+l}}{\partial x^k\partial y^l}F(\hat{x},0)\,,
$$
where $\hat{x}>0$.
Observe that $a_{k,0}=0$ for every $k$, since $F(x,0)\equiv 0$.
Moreover $a_{0,1}=\sin^2(\hat{x})/2$ and
$$
a_{k,1}=
\left\{
\begin{array}{ll}
\cos(2\hat{x})(-1)^{m+1}2^{2m-2}, & k=2m
\\
\sin(2\hat{x})(-1)^{m}2^{2m-1}, & k=2m+1\,.
\end{array}
\right.
$$
Therefore
$$
F(x,y)=\cos(2\hat{x})\sum_{m\geq 0}(-1)^{m+1}2^{2m-2}(x-\hat{x})^{2m} \frac{y}{(2m)!}+
\sin(2\hat{x})\sum_{m\geq 0}(-1)^{m}2^{2m-1}(x-\hat{x})^{2m+1} \frac{y}{(2m+1)!}+
$$
$$
+\frac y4+
y^2\sum_{k\geq 0,l\geq 2}\frac{a_{k,l}}{k!l!}(x-\hat{x})^k y^{l-2},
$$
that is,
\begin{eqnarray*}
F(x,y)&=&\frac y4[-\cos(2\hat{x})\cos(2(x-\hat{x}))+1+\sin(2\hat{x})\sin(2(x-\hat{x}))]
+
y^2\frac{a_{0,2}}{2}+\\
&+& y^3\sum_{l\geq 0}\frac{a_{0,l+3}}{(l+3)!}y^{l}
+
y^2(x-\hat{x})\sum_{k\geq 0,l\geq 0}\frac{a_{k+1,l+2}}{(k+1)!(l+2)!}(x-\hat{x})^{k} y^{l},
\end{eqnarray*}
where $$
a_{0,2}=\frac{\cos(\hat{x})\sin^4(\hat{x})}{4(\sin(\hat{x})-\hat{x}\cos(\hat{x}))}.
$$
We note that the last two series are convergent for $|x-\hat{x}|$ and $|y|$ sufficiently small, since the Taylor series of $F$ at $(\hat{x},0)$ is absolutely convergent.
Therefore, if $\eta_1=\hat{\eta}_1+\varepsilon_1$, one has
\begin{eqnarray*}
\frac{1}{\varepsilon^2}
F\left(\eta_1, \frac{\varepsilon}{\sin^2(\eta_1)}\right)&=&
\frac{1}{4\sin^2(\eta_1)\varepsilon}[-\cos(2\hat{\eta}_1)\cos(2\varepsilon_1)+1+\sin(2\hat{\eta}_1)\sin(2\varepsilon_1)]+
\frac{1}{\sin^4(\eta_1)}\frac{a_{0,2}}{2}+\\
&+&
\frac{\varepsilon}{\sin^6(\eta_1)}\sum_{l\geq 0}\frac{a_{0,l+3}}{(l+3)!}\frac{\varepsilon^l}{\sin^{2l}(\eta_1)}
+\frac{\varepsilon_1}{\sin^4(\eta_1)}\sum_{k\geq 0,l\geq 0}\frac{a_{k,l+3}}{k!(l+3)!}\varepsilon_1^k \frac{\varepsilon^l}{\sin^{2l}(\eta_1)}.
\end{eqnarray*}
When $\varepsilon, \varepsilon_1\to 0$, the first term is equivalent to 
$
\frac{1}{2\varepsilon}
$, the second one is equal to $\frac{\cos(\hat{\eta}_1)}{8(\sin(\hat{\eta}_1)-\hat{\eta}_1\cos(\hat{\eta}_1))}$
and the third and the fourth ones go to 0. This implies that
\begin{equation}\label{limA1}
\frac{1}{\varepsilon^2}
F\left(\eta_1, \frac{\varepsilon}{\sin^2(\eta_1)}\right)=
\frac{\cos(\hat{\eta}_1)}{8(\sin(\hat{\eta}_1)-\hat{\eta}_1\cos(\hat{\eta}_1))}
+ \frac{1}{2\varepsilon} +o(1),
\end{equation}
and the first term is positive.

\noindent
{\bf Case $A_2$.}
As in  case $A_1$, one can write the series expansion of $F(x,y)$. Since $\eta_2=\hat{\eta}_2+\varepsilon_2$, it holds 
\begin{eqnarray*}
\frac{1}{\varepsilon^2}
F\left(\eta_2, -\frac{\varepsilon}{\sin^2(\eta_2)}\right)&=&
 -\frac{1}{4\sin^2(\eta_2)\varepsilon}[-\cos(2\hat{\eta}_2)\cos(2\varepsilon_2)+1+\sin(2\hat{\eta}_2)\sin(2\varepsilon_2)] +\frac{1}{\sin^4(\eta_2)}\frac{a_{0,2}}{2}\\
&-&\frac{\varepsilon}{\sin^6(\eta_2)}\sum_{l\geq 0}\frac{a_{0,l+3}}{(l+3)!}\frac{(-\varepsilon)^l}{\sin^{2l}(\eta_2)}
+\frac{\varepsilon_2}{\sin^4(\eta_2)}\sum_{k\geq 0,l\geq 0}\frac{a_{k,l+3}}{k!(l+3)!}\varepsilon_2^k \frac{(-\varepsilon)^l}{\sin^{2l}(\eta_2)}.
\end{eqnarray*}
When $\varepsilon, \varepsilon_2\to 0$, the first term is equivalent to 
$
-\frac{1}{2\varepsilon}
$, the second one is equal to $\frac{\cos(\hat{\eta}_2)}{8(\sin(\hat{\eta}_2)-\hat{\eta}_2\cos(\hat{\eta}_2))}$
and the third and the fourth ones go to 0.
This implies that
\begin{equation}\label{limA2}
\frac{1}{\varepsilon^2}
F\left(\eta_2, \frac{\varepsilon}{\sin^2(\eta_2)}\right)=
\frac{\cos(\hat{\eta}_2)}{8(\sin(\hat{\eta}_2)-\hat{\eta}_2\cos(\hat{\eta}_2))}
- \frac{1}{2\varepsilon} +o(1)
\end{equation}
and the first term is positive.

\noindent
{\bf Case $B_1$.} 
The angle $\eta_1$ tends to zero, while $\theta_1=h^{-1}(h(\eta_1)+\frac{\varepsilon}{\sin^2(\eta_1)})$ converges to $h^{-1}(l_1)\neq 0$. 
Hence it holds
\begin{equation*}\label{limB1}
\frac{1}{\varepsilon^2}F\left(\eta_1, \frac{\varepsilon}{\sin^2(\eta_1)}\right)=
\frac{\eta_1}{\varepsilon^2}
\left[
\frac{\sin(\eta_1)}{\eta_1}\frac{h^{-1}(h(\eta_1)+\frac{\varepsilon}{\sin^2(\eta_1)})}{\sin(h^{-1}(h(\eta_1)+\frac{\varepsilon}{\sin^2(\eta_1)}))}-1
\right].
\end{equation*}
We observe that $\frac{\eta_1}{\varepsilon^2}\to +\infty$ and 
$$
\liminf_{\varepsilon\to 0} \frac{\sin(\eta_1)}{\eta_1}\frac{\theta_1}{\sin(\theta_1)}-1>0.
$$

\noindent
{\bf Case $B_2$.} 
In this case $\theta_2=h^{-1}(h(\eta_2)-\frac{\ep}{\sin^2(\eta_2)})$, so that, as in case $B_1$,
\begin{equation*}\label{limB2}
\frac{1}{\varepsilon^2}F\left(\eta_2, -\frac{\varepsilon}{\sin^2(\eta_2)}\right)=
\frac{\eta_2}{\varepsilon^2}
\left[
\frac{\sin(\eta_2)}{\eta_2}\frac{h^{-1}(h(\eta_2)-\frac{\varepsilon}{\sin^2(\eta_2)})}{\sin(h^{-1}(h(\eta_2)-\frac{\varepsilon}{\sin^2(\eta_2)}))}-1
\right].
\end{equation*}
As before, $\frac{\eta_2}{\varepsilon^2}\to +\infty$ and  
$$
\liminf_{\varepsilon\to 0} \frac{\sin(\eta_2)}{\eta_2}\frac{\theta_2}{\sin(\theta_2)}-1>0,
$$  since $\eta_2$ tends to zero, while $\theta_2$ tends to $h^{-1}(l_2)\neq 0$.

\noindent
{\bf Case $C_1$.}
Since $F$ is analytic, we can write 
$$
F(x,y)=\sum_{k,l}\frac{a_{k,l}}{k!l!}x^k y^l\,,\qquad\text{where } a_{k,l}=\frac{\partial^{k+l}}{\partial x^k\partial y^l}F(0,0).
$$
We need the exact value of some of the coefficients $a_{k,l}$.
Observe that $a_{k,0}=0$ for every $k$, since $F(x,0)\equiv 0$; as well, $a_{0,l}=0$ for every $l$, since $F(0,y)\equiv 0$.
Moreover $a_{0,1}=0$ and for $k\geq 1$
$$
a_{k,1}=
\left\{
\begin{array}{ll}
0, & k=2m+1
\\
(-1)^{m+1}2^{2m-2}, & k=2m
\end{array}
\right.
$$
and   $a_{1,2}=3/4$.
Hence $F$ can be written as
$$
F(x,y)=y\sum_{k\geq 1}(-1)^{k+1}\frac{2^{2k-2}}{(2k)!}x^{2k}+
y^2\sum_{k\geq 0,l\geq 0}\frac{a_{k,l+2}}{k!(l+2)!}x^k y^{l}=
\frac{y \sin^2(x)}{2}+
y^2\sum_{k\geq 1,l\geq 0}\frac{a_{k,l+2}}{k!(l+2)!}x^k y^{l},
$$
that is,
$$
F(x,y)=\frac{y \sin^2(x)}{2}+
xy^2\frac 38 +
xy^3 \sum_{l\geq 0}\frac{a_{1,l+3}}{(l+3)!} y^{l}
+
x^2 y^2\sum_{k\geq 0,l\geq 0}\frac{a_{k+2,l+2}}{(k+2)!(l+2)!}x^k y^{l}.
$$
We observe that the last two terms are convergent for $|x|$ and $|y|$ sufficiently small, since the Taylor series of $F$ at $(0,0)$ is absolutely convergent.
Therefore
\begin{eqnarray*}
&\displaystyle \frac{1}{\varepsilon^2}F\left(\eta_1,\frac{\varepsilon}{\sin^2(\eta_1)}\right)=&\\
&=\displaystyle {\frac{1}{2\varepsilon}}
+\frac{\eta_1}{\sin^4(\eta_1)}\left[\frac 38 +
\frac{\varepsilon}{\sin^2(\eta_1)}  \sum_{l\geq 0}\frac{a_{1,l+3}}{k!(l+3)!} \frac{\varepsilon^l}{\sin^{2l}(\eta_1)}
+
{\eta_1}\sum_{k\geq 0,l\geq 0}\frac{a_{k+2,l+2}}{(k+2)!(l+2)!} \frac{\eta_1^k\varepsilon^l}{\sin^{2l}(\eta_1)}\right],&
\end{eqnarray*}
and the last term tends to $\frac 38$. This implies that
\begin{equation*}\label{limC1}
\frac{1}{\varepsilon^2}F\left(\eta_1,\frac{\varepsilon}{\sin^2(\eta_1)}\right)= {\frac{1}{2\varepsilon}}+\frac 38\frac{\eta_1}{\sin^4(\eta_1)}\,+o(1).
\end{equation*}

{\bf Case $C_2$.}
In this case, $\theta_2=h^{-1}(h(\eta_2)-\frac{\ep}{\sin^2\eta_2})$. Using the same argument as in case $C_1$, we have
\begin{equation*}\label{limC2} 
 \frac{1}{\varepsilon^2}F\left(\eta_2,-\frac{\varepsilon}{\sin^2(\eta_2)}\right)= {\frac{-1}{2\varepsilon}}+\frac38 \frac{\eta_2}{\sin^4\eta_2} +o(1)\,.
\end{equation*}

\noindent
{\bf Case $D_1$.}
We claim that
$$
h(\pi-\alpha)\geq \frac{\pi}{\alpha^2}\geq h(\pi-\alpha-\alpha^2)\,,\,\,\,\,\,\,0\leq \alpha<0.9.
$$
This can be easily proved recalling that 
\begin{equation}\label{sinus}
t-\frac{t^3}{6}\leq \sin(t)\leq t-\frac{t^3}{6}+\frac{t^5}{120}.
\end{equation}
Indeed the first inequality to prove is equivalent to $\alpha^2(\pi-\alpha+\sin(2\alpha)/2)-\pi \sin^2(\alpha)\geq 0$
and by (\ref{sinus}) a bound from below of the left hand side is $\frac{\pi}{3}\alpha^4-\frac 23 \alpha^5-\frac{2\pi}{45}\alpha^6+\frac{\pi}{360}\alpha^8-\frac{\pi}{14400}\alpha^{10}$ which is positive for $\alpha<1$.
On the other hand, the second inequality is equivalent to $\alpha^2(\pi-\beta+\sin(2\beta)/2)-\pi \sin^2(\beta)\leq 0$, where $\beta=\alpha+\alpha^2$ and by (\ref{sinus}) a bound from above of the left hand side is $\alpha^2\left(\pi-\frac 23 \beta^3+\frac{2}{15}\beta^5\right)-\pi(\beta-\frac{\beta^3}{6})^2$ which is negative for $\alpha<0.9$.
\\
Let us set $\frac{\pi}{\alpha^2}=h(\eta_1)+\frac{\varepsilon}{\sin^2(\eta_1)}$. The hypotheses on $\eta_1$ imply that $\alpha\to 0$. Therefore
$$
\lim_{\varepsilon\to 0}  F\left(\eta_1,\frac{\varepsilon}{\sin^2(\eta_1)}\right)=\lim_{\varepsilon\to 0}  \sin(\eta_1)\frac{h^{-1}(\pi/\alpha^2)}{\sin(h^{-1}(\pi/\alpha^2))}-\eta_1=
\lim_{\varepsilon\to 0}  \sin(\eta_1)\frac{\pi-\alpha}{\sin(\alpha)}-\eta_1.
$$
Since $\alpha$ is equivalent to $\frac{\sqrt{\pi}}{\sqrt{\varepsilon}}\sin(\eta_1)$, one has
$$
 F\left(\eta_1,\frac{\varepsilon}{\sin^2(\eta_1)}\right)=
\frac{{\pi-\alpha}}{\sqrt{\pi}}\sqrt{\varepsilon}-\eta_1 +o(\sqrt{\varepsilon}).
$$
The hypotheses on $\eta_1$ imply that 
\begin{equation*}\label{limD1}
F\left(\eta_1,\frac{\varepsilon}{\sin^2(\eta_1)}\right)=
 {\sqrt{\pi}}\sqrt{\varepsilon} +o(\sqrt{\varepsilon}).
\end{equation*}

\noindent
{\bf Case $D_2$}.
Set $-\frac{\pi}{\alpha^2}=h(\eta_2)-\frac{\ep}{\sin^2\eta_2}$. Since $h(t)$ is an odd function, we have
$$
F\left(\eta_2,-\frac{\ep}{\sin^2\eta_2}\right)= \sin\eta_2 \frac{h^{-1}(\frac{\pi}{\alpha^2})}{\sin({h^{-1}(\frac{\pi}{\alpha^2})})}-\eta_2,
$$ 
which is analogous to  case $D_1$. Hence 
\begin{equation*}\label{limD2}
F\left(\eta_2,-\frac{\ep}{\sin^2\eta_2}\right)=\sqrt{\pi\ep}+o(\sqrt{\varepsilon}).
\end{equation*}

\bigskip

We are now able to compute
 $\displaystyle 
\liminf_{\varepsilon\to 0} \mathcal{F}(\Omega^*_\varepsilon),
$ 
by observing that
$$
\liminf_{\varepsilon\to 0}\mathcal{F}(\Omega^*_\varepsilon)=
\liminf_{\varepsilon\to 0}\frac{\delta(\Omega^*_\varepsilon)}{\lambda^2(\Omega_\varepsilon^*)}\geq
\liminf_{\varepsilon\to 0}\frac{\delta(\Omega^*_\varepsilon) \pi^2}{16\varepsilon^2}
\geq\frac{\pi}{8}\liminf_{\varepsilon\to 0} 
\left[\frac{1}{\varepsilon^2}F\left(\eta_1, \frac{\varepsilon}{\sin^2(\eta_1)}\right)
+
\frac{1}{\varepsilon^2}F\left(\eta_2, \frac{-\varepsilon}{\sin^2(\eta_2)}\right)\right].
$$
The technique consists in combining the behaviour of the function $F(x,y)$ for $x=\eta_i$ and $y=\pm\frac{\ep}{\sin^2\eta_i}$. This obviously depends on the behaviour of the angles $\eta_i,\theta_j$ ($i,j\in\{1,..,4\}$). Hence all the different possible situations  have to be considered.
In all the cases, except the case $(A_1,A_2)$,   $\liminf_{\varepsilon\to 0}\mathcal{F}(\Omega^*_\varepsilon)$ is infinite. 
In the case $(A_1,A_2)$, the liminf is finite, but larger than $\frac{\pi}{8(4-\pi)}$. Indeed thanks to (\ref{limA1}) and (\ref{limA2}), it holds 
\begin{equation}\label{eq=}
 \liminf_{\varepsilon\to 0}\frac{\delta(\Omega_\varepsilon^*)}{\lambda^2(\Omega_\varepsilon^*)}
\geq \frac{\pi}{8}
\left[\frac{\cos(\hat{\eta}_1)}{8(\sin(\hat{\eta}_1)-\hat{\eta}_1\cos(\hat{\eta}_1))}+
\frac{\cos(\hat{\eta}_2)}{8(\sin(\hat{\eta}_2)-\hat{\eta}_2\cos(\hat{\eta}_2))}\right],
\end{equation}
since the terms in $\varepsilon$ cancel each other. Now, by the convexity of the function $x\mapsto 
\frac{\cos x}{8(\sin x- x\cos x)}$, it is easy to see that
the minimum of the above function of $(\hat{\eta}_1,\hat{\eta}_2)$ is attained for
$(\hat{\eta}_1,\hat{\eta}_2)=(\pi/4,\pi/4)$, that is,
\begin{equation}\label{A1A2}
\lim_{\varepsilon\to 0}\frac{\delta(\Omega_\varepsilon^*)}{\lambda(\Omega_\varepsilon^*)^2}
\geq
\frac{\pi}{4}
\frac{\cos(\pi/4)}{8(\sin(\pi/4)-\pi/4\cos(\pi/4))}=\frac{\pi}{8(4-\pi)}>0.44.
\end{equation}
\end{proof}

In the following proposition we prove that  the symmetrization of Definition \ref{rearrangement} makes $\mathcal{F}$  asymptotically decreasing. This is the key point of our approach.
\begin{prop}\label{prop_simmetrizzazione}
For every $\alpha>0$ there exists $\beta>0$ such that for every $\Omega$ transversal to an optimal ball $B$ with $\lambda(\Omega)\leq \beta$, one has 
$
\mathcal{F}(\Omega^*)
\leq
\mathcal{F}(\Omega)
+\alpha
$.
\end{prop}
\newcommand{\onda}{(10,0) to[out=90,in=300] (4,4) to[out=120,in=0] (1,10) to[out=180,in=45] (-4,6) to[out=225,in=30] (-10,8) to[out=210,in=90] (-8,0) to[out=270,in=120] (-11,-4) to[out=300, in=135] (-5,-3) to[out=315,in=180] (0,-6) to[out=0,in=240](2,-2) to[out=60,in=270](10,0)}
\newcommand{\cerchio}{(-2,2)circle(8)}
	\newcommand{\figuraEp}{
		\begin{scope}[rotate=45]
		\fill[gray!30] \onda;	
	 \fill[gray!30]\cerchio;
		\begin{scope}
		\clip \onda;
		\fill[white]\cerchio ;
		\end{scope}
			\draw[thick, blue] \onda; \draw[thick] \cerchio;
			\fill (-2,2)++(-22:8)node[below right]{$A_1$} circle(2pt);\fill (-2,2)++(4:8)node[above]{$A_2$} circle(2pt);\fill (-2,2)++(53:8)node[above right]{$A_3$} circle(2pt);	\fill (-2,2)++(77:8)node[above left]{$A_4$} circle(2pt); \fill (-2,2)++(136:8)node[above left]{$A_5$} circle(2pt);	\fill (-2,2)++(165:8) circle(2pt);\fill (-2,2)++(207:8) circle(2pt);\fill (-2,2)++(223:8) circle(2pt);\fill (-2,2)++(165:8) circle(2pt);\fill (-2,2)++(277:8) node[below right]{$A_{2p-1}$} circle(2pt);\fill (-2,2)++(293:8) node[below right]{$A_{2p}$}circle(2pt);
		\draw[<-, blue] (9,0)--(11,2)node[above]{area $\ep_1$};
		\draw[<-, blue ] (4,6)--(6,7)node[above]{area $\ep_2$};
		\draw[<-, blue] (1.3,9.6)--(2,11)node[above]{area $\ep_3$};	
		\draw[<-, blue] (-4,8)--(-4,11)node[above]{area $\ep_4$};
		\draw[<-,blue] (2,-4)--(5,-7)node[right]{area $\ep_{2p}$};
		\draw (-13,0)node{$\cdots$};\draw[blue] (-9,11)node{$\cdots$};
		\fill[gray!30] (-2,2)--(-2,2)++(-45:2)arc (-45:4:2)--(-2,2);
		\fill[gray!70] (-2,2)--(-2,2)++(-45:1.5)arc(-45:-22:1.5)--(-2,2);
		\draw[red] (-2,2)node[below]{$O$};
		\draw[red, dashed] (-2,2)++(-22:8)--(-2,2);\draw[red, dashed] (-2,2)++(4:8)--(-2,2);\draw[red, dashed](-2,2)++(53:8)--(-2,2);
		\draw[red, thick] (-2,2)++(-45:1.5)arc (-45:-22:1.5);\draw[<-, red] (-2,2)++(-40:1)--(-2,0.5)node[below]{$\varphi_1$};
		\draw[thick, red] (-2,2)++(-45:2)arc (-45:4:2); \draw[<-, red] (-2,2)++(-10:1.7)--(0,-2)node[below]{$\varphi_2$};
		\draw[thick, red] (-2,2)++(-22:2.5)arc (-22:4:2.5); \draw[red] (-2,2)++(0:2.5) node[right]{\small{$2\eta_1$}};
 	\draw[red] (-2,2)++(-45:17)node[above]{$x$}--(-2,2); \draw[red] (-2,2)++(135:15)--(-2,2);
		\end{scope}	
	}

\begin{proof}
Let $0\leq \varphi_1< \varphi_2<...< \varphi_{2p}\leq 2\pi$ be the angles determined by the intersection points $A_1,...,A_{2p}$ 
between $\partial \Omega$ and $\partial B$ defined as $\varphi_i=(Ox,OA_i)$ as shown in Figure \ref{fig-epi}.
Define  $\eta_j=(\varphi_{j+1}-\varphi_{j})/2$.

\begin{figure}[h]
	\centering
	\begin{tikzpicture}[x=3mm,y=3mm]
	\figuraEp
	\end{tikzpicture}
	\caption{The points $A_i$, the angles $\varphi_i$ and the areas $\ep_i$ in the proof of Proposition \ref{prop_simmetrizzazione}.}\label{fig-epi}
\end{figure}
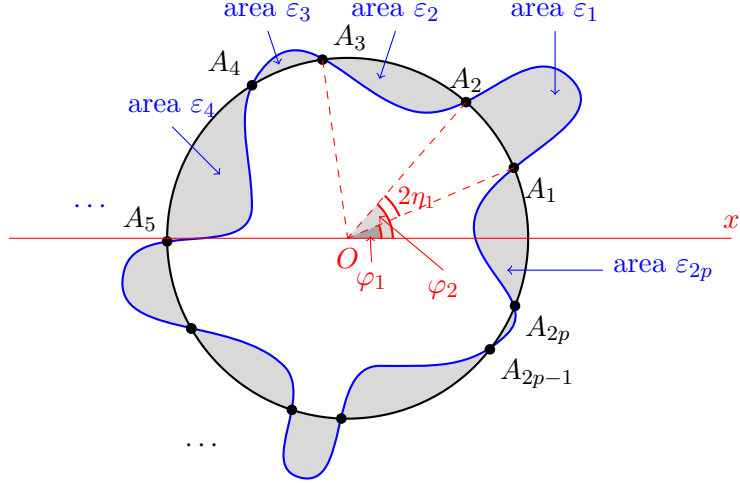

Let $\Gamma_{OUT}=(A_1, A_2)\cup (A_3, A_4)\cup...\cup(A_{2p-1},A_{2p})\subset \partial B$ 
and 
$\Gamma_{IN}=(A_2, A_3)\cup (A_4, A_5)\cup...\cup (A_{2p}, A_1)\subset \partial B$.
According to the figure, let us denote by $\ep_i$, for $i=1,...,2p$, the area of  the connected component of $\Omega\Delta B$ whose boundary contains the points $A_i, A_{i+1}$. 
Using the solution of the Dido problem on each of these connected components (where the arc $A_i, A_{i+1}$
is fixed), we can replace all components of $\Omega\Delta B$
by a set of same measure, bounded by two arcs of circle, then
it holds
\begin{equation}\label{291}
\mathcal{F}(\Omega) \ge \dfrac{\pi^2}{4\ep^2}\left( \frac{\sum_{i=1}^{2p} \mathcal{H}^1(C_i)}{2\pi}-1 \right),
\end{equation}
where $4\ep=|\Omega\Delta B|=\sum_{i=1}^{2p}\ep_i$ and $C_i$ are arcs of circles with end-points $A_i, A_{i+1}$ such that the area of the region enclosed by the arc $C_i$ and the ball $B$ equals $\ep_i$.
Since the area of $\Omega$ is $\pi$, we have  
\begin{equation}\label{epi}
2\varepsilon =\sum_{i=1}^p \varepsilon_{2i-1}=\sum_{i=1}^p \varepsilon_{2i}.
\end{equation}

We are going to minimize the perimeter of $\Omega$ over each of the sets  $\Omega\setminus B$ and $B\setminus \Omega$, separately. We will prove that the minimizer is $\Omega^*$, in both cases. This will imply that the minimizer of  $\mathcal{F}$ is $\Omega^*$.

Notice that the set $\Omega$ satisfies the following conditions:
\begin{equation}\label{constraints}
\begin{cases}
\sum_{i=1}^p(\varphi_{2i}-\varphi_{2i-1})=\mathcal{H}^1(\Gamma_{OUT})=\gamma_{OUT},\\
\sum_{i=1}^{p-1}(\varphi_{2i+1}-\varphi_{2i})+\varphi_{1}-\varphi_{2N}=\mathcal{H}^1(\Gamma_{IN})=\gamma_{IN},\\
\sum_{i=1}^p (\cos(\varphi_{2i})-\cos(\varphi_{2i-1}))=0,\\
\sum_{i=1}^p (\sin(\varphi_{2i})-\sin(\varphi_{2i-1}))=0,
\end{cases}
\end{equation}
where the last two constraints are a consequence of Proposition \ref{cdt_optimalite} and
$$
\lambda(\Omega)=\dfrac{4\varepsilon}{\pi}.
$$

Let us consider the set $\Omega\setminus B$.
Let $f(x)=\frac{\cos x}{\sin x-x\cos x}$; according to the analysis of Case $A_1$ in the proof of Theorem \ref{casABCD},  we have to study the minimization problem
\begin{equation}\label{minimizzazioneasintotica}
\min\left\{\sum_{i=1}^p \left[\frac{\varepsilon_{2i-1}}{2}+\frac{\varepsilon_{2i-1}^2}{8} f(\eta_{2i-1})+o(\varepsilon_{2i-1}^2)\right]\right\}\,,
\end{equation}
under the constraints in (\ref{epi}), (\ref{constraints}).
Instead of solving the complete minimization problem we are going to consider only the first two terms of the developement in (\ref{minimizzazioneasintotica})
that is, we minimize the function
$$
G(\varepsilon_1,\varepsilon_3,...,\varepsilon_{2p-1},\eta_1, \eta_3,...,\eta_{2p-1})=\sum_{i=1}^p {\varepsilon_{2i-1}^2} f(\eta_{2i-1}),
$$ 
under  constraints  (\ref{epi}) and (\ref{constraints}).

Observe that the $\varepsilon_i$'s are in a compact set. Let us first solve the minimization problem with respect to the $\varepsilon_i$.
If we compute the derivative of $G$ with respect to $\varepsilon_i$, by  constraint (\ref{epi}), we get, for every $i=1,..,p$
\begin{equation}\label{firstlambda}
2\varepsilon_{2i-1}f(\eta_{2i-1})=\lambda_0,
\end{equation}
where $\la_0$ is a Lagrange multiplier, that is,
\begin{equation}
\label{secondlambda}
\lambda_0=\frac{4\varepsilon}{\sum_{i=1}^p\frac{1}{f(\eta_{2i-1})}}.
\end{equation}
If we replace $\varepsilon_{2i-1}$ in the expression of  $G$,  problem (\ref{minimizzazioneasintotica}) reduces to find
$
\max \sum_{i=1}^p\frac{1}{f(\eta_{2i-1})}
$, that is,  
\begin{equation}\label{secondaminimizzazioneasintotica}
\max \sum_{i=1}^p \tan \eta_{2i-1}\,,
\end{equation}
due to  constraint (\ref{constraints}).

We are going to prove that $\eta_1, \eta_3..., \eta_{2p-1}$ are in a compact set contained in $[0,\frac{\pi}{2})$
and the existence of a maximizer for (\ref{secondaminimizzazioneasintotica}) will follow.
First of all, we observe  that $\eta_{2i-1}<\frac{\pi}{2}$ for every $i\in \{1,...,p\}$. Indeed, if $\eta_{2i-1}>\frac{\pi}{2}$ for some $i$, then $\lambda_0$ would be negative by (\ref{firstlambda}). 
If $\eta_{2i-1}=\frac{\pi}{2}$ for some $i$,  then $\lambda_0=0$ and we would find {one arc}. 
Notice that this is a contradiction, as observed in Remark \ref{remarknumberintersection}.

Since the problem is rotations invariant, we can assume that $(A_1,A_2)$ is the longest arc and $A_1\equiv(x_1,y)$, $A_2\equiv(-x_1,y)$, with $y>0$, that is, $\varphi_2=\pi-\varphi_1$.  Let us assume that 
$$
\frac{\pi}{2}< \varphi_3\leq  \varphi_4\leq ...\leq \varphi_m<\pi,
\qquad0<\varphi_{q+1}\leq \varphi_{q+2}\leq ...\leq \varphi_{2p}<\frac{\pi}2\,,
$$
for some $m, q \in \N$. 
We claim that if $\mathcal{H}^1(\Gamma_{IN}\cap \{y>0\})\geq 
{\mathcal{H}^1(\Gamma_{IN})}/{2}$, then $\varphi_1\geq {\mathcal{H}^1(\Gamma_{IN})}/{4}$. 
Indeed, constraints (\ref{constraints}) imply that 
\begin{equation}\label{sum_cos}
-2 \cos \varphi_1= \cos \varphi_2 - \cos \varphi_1=
(\cos \varphi_3-\cos \varphi_4)+...+(\cos \varphi_{2p-1}-\cos \varphi_{2p}).
\end{equation}
We are going to estimate from below the right hand side. We will divide our analysis according to the parity of $m$ and $q$. 

If $m$ is even, then in the right hand side of (\ref{sum_cos})
all the terms involving indices less or equal to $m$ are positive.
Assume that $m$ is odd.  Then in the right hand side of (\ref{sum_cos})
the terms $\cos \varphi_3-\cos \varphi_4, ..., \cos \varphi_{2j-1}-\cos \varphi_{2j}$, ... , 
$\cos \varphi_{m-2}-\cos \varphi_{m-1}$ are positive. 
We rewrite $\cos \varphi_m-\cos \varphi_{m+1}=(\cos \varphi_m-\cos \pi)+(\cos \pi-\cos \varphi_{m+1})$. 
The first term is positive and the second one will be treated later.
In conclusion, the sum of these terms is greater than $-1$. 

The points in the first quadrant $\{x\geq 0\}\cap \{y\geq 0\}$ will be treated in the same way.
For the other points, we have to estimate the measure of the projection on the $x$ line of $\Gamma_{IN}\cap \{y\leq 0\}$.
To do that, we observe that if $\Gamma$ is an arc in $\{y\leq 0\}$, then the measure of its projection $P_x$ on 
the $x$ line is greater or equal to 
$
2\left(1-\cos \frac{\mathcal{H}^1(\Gamma)}{2}\right).
$
Since we are assuming that
 $\mathcal{H}^1(\Gamma_{IN}\cap \{y\leq 0\})\leq 
\frac{\mathcal{H}^1(\Gamma_{IN})}{2}$, then 
$$
\mathcal{H}^1(P_x(\Gamma_{IN}\cap \{y\leq 0\}))\geq 
2\left[1-\cos \left(\frac{\mathcal{H}^1(\Gamma_{IN}\cap \{y\leq 0\})}2\right)\right]\geq
2\left(1-\cos \frac{\mathcal{H}^1(\Gamma_{IN})}4\right).
$$ 
The above  estimates  imply that $\cos\varphi_1\leq \cos(\frac{\mathcal{H}^1(\Gamma_{IN})}{4})$ and then 
$$
\eta_i\leq \eta_1=\frac{\varphi_2-\varphi_1}{2}\leq \frac{\pi}{2}-\frac{\mathcal{H}^1(\Gamma_{IN})}{4}.
$$

If $\mathcal{H}^1(\Gamma_{IN}\cap \{y>0\})\leq 
{\mathcal{H}^1(\Gamma_{IN})}/{2}$, then $\varphi_1\geq {\mathcal{H}^1(\Gamma_{IN})}/{4}$. 
Indeed, we get the same estimate as in the previous case:
$$
\mathcal{H}^1(P_x(\Gamma_{IN}\cap \{y\leq 0\}))\geq 
2\left[1-\cos \left(\frac{\mathcal{H}^1(\Gamma_{IN}\cap \{y\leq 0\})}2\right)\right]\geq
2\left(1-\cos \frac{\mathcal{H}^1(\Gamma_{IN})}4\right).
$$ 

We now study problem (\ref{secondaminimizzazioneasintotica}). Let us write the optimality conditions,
we get
\begin{eqnarray}\label{derivation}
-\frac 12(1+\tan^2 \eta_i)&=&-\lambda_0+\lambda_1 \sin \varphi_{2i-1}-\lambda_2\cos \varphi_{2i-1},
\\ \nonumber
\frac 12(1+\tan^2 \eta_i)&=&\lambda_0-\lambda_1 \sin \varphi_{2i}+\lambda_2 \cos \varphi_{2i},
\end{eqnarray}
where $\la_0,\la_1,\la_2$ are Lagrange multipliers.

Assume  now $(\lambda_1,\lambda_2)\neq (0,0)$. Let $a=\sqrt{\lambda_1^2+\lambda_2^2}$. Then $\lambda_1=a\sin\theta_0$ and 
$\lambda_2=a\cos\theta_0$. Summing up in (\ref{derivation}),
we get for any $i$ $\cos(\theta_0+\varphi_{2i})=\cos(\theta_0+\varphi_{2i-1})$. The only possibility is $\theta_{0}+\varphi_{2i}=-\theta_{0}-\varphi_{2i-1}$, that is, the $\eta_i$'s are equal.
Therefore the functional $\mathcal{F}$ is $k\tan(\frac{\mathcal{H}^1(\Gamma_{OUT})}{2k})$, for some $k\in \N$. 
The maximum is attained at the minimal possible value of $k$, which is 2 (notice, indeed, that $k=1$ is impossible due to the constraints).
For $k=2$ we obtain a geometric configuration which coincides with the symmetrized set $\Omega^*$.

Assume   $(\lambda_1,\lambda_2)=(0,0)$. For every $i\in \{1,...,p\}$, $\tan^2 \eta_{2i-1}=2\lambda_0-1$. Therefore $\eta_{2i-1}=\frac{\pi}{2}\pm h$.
 Since $\sum_{i=1}^p \eta_{2i-1}<\pi$, there exist $l$ elements $\eta_i$ equal to $\frac{\pi}{2}-h$ and either 0 or 1 element equal to  $\frac{\pi}{2}+h$.
The last case is impossible, since $l(\frac{\pi}{2}-h)+\frac{\pi}{2}+h<\pi$ is in contradiction with $h\leq \frac{\pi}{2}$.
Therefore we have 
$$
\frac{\sum_{i=1}^{2p}\mathcal{H}^1(C_i)}{2\pi}-1=l\tan\left(\frac{\mathcal{H}^1(\Gamma_{OUT})}{2l}\right),
$$
 as seen in the previous case.
We thus find  that the minimum is attained for $l=2$ and a similar conclusion holds true.

The analysis of $B\setminus\Omega$ is  analogous. Therefore $\Omega^*$ is a minimizer of (\ref{secondaminimizzazioneasintotica}). 

Hence, in every cases, by (\ref{minimizzazioneasintotica}) it holds
$$
\sum_{i=1}^{2p}\mathcal{H}^1(C_i)\ge P(\Omega^*) -\tilde{\alpha},
$$
for some $\tilde{\alpha}>0$.
By coupling the above inequality  with (\ref{291}), and recalling that $\la(\Omega^*)=\la(\Omega)=\frac{4\ep}{\pi}$, we obtain that   
\begin{equation}\label{292}
\mathcal{F}(\Omega)\ge\mathcal{F}(\Omega^*) -\alpha
\end{equation}
for some $\alpha>0$.

\end{proof}


We now describe the possible behaviour of any sequence converging to a ball. A consequence of the previous results is
\begin{cor}\label{corfond}
Let $\{\Omega_{\varepsilon} \}_{\ep>0}$ be a sequence of sets converging to a ball $B$ such that  
$|B\Delta \Omega_{\varepsilon} |= {4\varepsilon}$.
Then
\begin{equation}\label{mino}
\liminf_{\varepsilon\to 0}\mathcal{F}({\Omega_{\varepsilon}}) 
\geq  \frac{\pi}{8(4-\pi)}.
\end{equation}
\end{cor}
\begin{proof}
	Let $\alpha>0$ and let $\beta$ be the corresponding value to $\alpha/2$ given by  Proposition \ref{prop_simmetrizzazione}. 
Let $B_{\varepsilon}$ be an optimal ball for $\Omega_{\varepsilon}$ so that $\lambda(\Omega_{\varepsilon})\leq {4\varepsilon}/{\pi}$.
We  choose $\ep$ sufficiently small such that $\la(\Omega_\ep)<\beta/2$.
We can modify  $\Omega_{\varepsilon}$ into a transversal set $\widehat{\Omega}_{\varepsilon}$ to its optimal ball $B_{\varepsilon}$ such that
\begin{eqnarray*}
&|\widehat{\Omega}_{\varepsilon}|=\pi,&\\
&\displaystyle \left|\frac{\delta(\widehat{\Omega}_{\varepsilon})}{\lambda^2(\widehat{\Omega}_{\varepsilon})}-\frac{\delta(\Omega_{\varepsilon})}{\lambda^2(\Omega_{\varepsilon})}\right|\le \frac\alpha 2,&\\
&\lambda(\widehat{\Omega}_{\varepsilon})\leq \beta.&
\end{eqnarray*}
Since $\widehat{\Omega}_{\varepsilon}$ is transversal to $B_\varepsilon$,
by Proposition \ref{prop_simmetrizzazione} one has
$$
\frac{\delta(\widehat{\Omega}_{\varepsilon}^*)}{\lambda^2(\widehat{\Omega}_{\varepsilon}^*)}-\frac{\delta(\widehat{\Omega}_{\varepsilon})}{\lambda^2(\widehat{\Omega}_{\varepsilon})}\leq \frac\alpha 2,
$$
and summing up we get
$$
\frac{\delta(\widehat{\Omega}_{\varepsilon}^*)}{\lambda^2(\widehat{\Omega}_{\varepsilon}^*)}\leq \frac{\delta({\Omega_{\varepsilon}})}{\lambda^2({\Omega_{\varepsilon}})}+ \alpha. 
$$
By Theorem \ref{casABCD}, one has
$$
\liminf_{\varepsilon\to 0}\frac{\delta({\Omega_{\varepsilon}})}{\lambda^2({\Omega_{\varepsilon}})} +\alpha\geq\frac{\pi}{8(4-\pi)},
$$
and this entails the result, as $\alpha$ is arbitrary.
\end{proof}

We underline that, although inequality (\ref{mino}) is sufficient  to prove that the sequences converging to the ball are not ``competitive'',  we prove also
that the value $\frac{\pi}{8(4-\pi)}$ is sharp, in the sense that there exists a sequence $\{\Omega_n\}$ such that $\mathcal{F}(\Omega_n)$ 
gives this value at the limit. For that purpose, we need a preliminary result for optimal balls
which has, however, its own interest.


\begin{prop}\label{centre_optimizing_ball}
	Assume that  $\Omega\subset \R^2$ has a symmetry axis $\Pi$ and  is convex in the perpendicular direction.  
	Then there exists an optimal ball  centered  on $\Pi$.\\
	If, moreover, the domain is transversal to its optimal balls and it does not contain segments which are orthogonal to $\Pi$, then all the optimal balls are  centered  on $\Pi$.
\end{prop}
\begin{proof}
	We first prove the result for a transversal domain $\Omega$ to an optimal ball and we assume 
	that its boundary does not contain segments which are orthogonal to the axis of symmetry.
	We can assume that $\Pi$ lies on the $x$ axis. The boundary of $\Omega$ will then be given by $y=\pm f(x)$ for some function $f:\R\to\R$ such that $f\geq 0$ in an interval $[a,b]$, and $f(x)\equiv 0$ if $x\le a$ or $x\ge b$.
	Set 
	$$
	\psi (M)=\psi (x,y)=|\Omega \Delta B_M|,
	$$ 
	where $M=(x,y)$ and $B_M$ denotes the unitary ball of center $M$.
	Let $M_1, M_2, ...M_{2p}$ be the intersection points between $\bd\Omega$ and $\bd B_M$ in an counter-clockwise order such that the boundary of $\Omega$ comes into the disk following the standard orientation of the curve (see Figure \ref{Mi-figure} (a)).
	By Lemma \ref{lemmapsi} we have
	\begin{equation*}
		\frac{\partial \psi}{\partial y}=-2\Big(x_1+x_3+...+x_{2p-1} - (x_2+x_4+...+ x_{2p})\Big)\,.
	\end{equation*}
	Hence in the case where $B_M$ is an optimal ball for $\Omega$, we have that $x_1+x_3+...+x_{2p-1} - (x_2+x_4+...+ x_{2p})=0$. 
	
	In particular we denote by $u_i$ the $x$-coordinate of the point $M_i$ belonging to the half plane $\Pi^+=\{(x,y)\ :\ y\ge 0\}$ and by $z_j$ the $x$-component of the point $M_j$ belonging to $\Pi^-=\{(x,y)\ :\ y<0\}$.
	For simplicity we can assume $b\ge\sqrt{1-t^2}$  in such a way that $u_1>u_2>...>u_k$ and $z_{k+1}<z_{k+2}<...<z_{2p}$. The case $u_{2p}>u_1>u_2>...>u_k$ corresponds to Figure \ref{Mi-figure} (b) and can be treated in an analogous way.
	
	\newcommand{\Mi}{
		\begin{tikzpicture}[x=3mm,y=3mm]
		\draw[->] (-11,0)--(6,0)node[right]{$x$}; 	
		\draw[thick] (5,0) to[out=90,in=0] (0,6) to[out=180,in=30] (-5,3) to[out=210,in=90] (-10,0) to[out=270,in=150] (-5,-3) to[out=330,in=180] (0,-6) to[out=0,in=270] (5,0);
		\fill[black] (5,0)node[above right]{$b$} circle(2pt); \fill[black] (-10,0)node[above left]{$a$} circle(2pt);
		\draw[blue] (-2,0.5) circle(6);
		\fill[red] (0.5,6)node[above]{$M_1$} circle(2pt);\fill[red] (-7.8,2)node[above]{$M_2$} circle(2pt);\fill[red] (-7.2,-2.25)node[below]{$M_3$} circle(2pt);\fill[red] (-1.98,-5.5)node[below]{$M_4$} circle(2pt);
		\draw[red, dashed] (0.5,6)--(0.5,0)node[above right]{$u_1$};\draw[red, dashed]  (-7.8,2)--(-7.8,0)node[above left]{$u_2$};
		\draw[red, dashed] (-7.2,-2.25)--(-7.2,0) node[above right]{$z_3$};\draw[red, dashed] (-1.98,-5.5)--(-1.98,0)node[above left]{$z_4$};
		\draw (4,4)node[above]{$\Omega$};
		\begin{scope}[scale=1.7]
		\draw (0,-5)node[below]{(a)};
		\begin{scope}[xshift=4cm]
		\draw[->] (-6,0)--(6,0)node[right]{$x$}; 	
		\draw[thick] (4,0) to[out=90,in=0] (0,4.5) to[out=180,in=60] (-2,2) to[out=240,in=90] (-4,0) to[out=270,in=120] (-2,-2) to[out=300,in=180] (0,-4.5) to[out=0,in=270] (4,0);
		\fill[black] (4,0)node[above right]{$b$} circle(1pt); \fill[black] (-4,0)node[above left]{$a$} circle(1pt);
		\draw[blue] (0.5,0.5) circle(3.7);
		\fill[red] (1.75,4)node[above right]{$M_6$} circle(1.2pt);\fill[red] (-1.1,3.85)node[above left]{$M_1$} circle(1.2pt);\fill[red] (-3.1,1.15)node[above]{$M_2$} circle(1.2pt);\fill[red] (-2.8,-1.25)node[below left]{$M_3$} circle(1.2pt);\fill[red] (-1.75,-2.4)node[below left]{$M_4$} circle(1.2pt);\fill[red] (3.85,-1)node[right]{$M_5$} circle(1.2pt);
		\draw[red, dashed] (1.75,4)--(1.75,0)node[above left]{$u_6$};\draw[red, dashed]  (-1.1,3.85)--(-1.1,0)node[below]{$u_1$};\draw[red, dashed]  (-3.1,1.15)--(-3.1,0)node[below]{$u_2$};\draw[red, dashed]  (-2.8,-1.25)--(-2.8,0)node[above]{$z_3$};
		\draw[red, dashed] (-1.75,-2.4)--(-1.75,0) node[above]{$z_4$};\draw[red, dashed] (3.85,-1)--(3.85,0)node[above left]{$z_5$};
		\draw (2,-4.5)node[right]{$\Omega$};
		\draw (0,-5)node[below]{(b)};
		\end{scope} 
		\end{scope}
		\end{tikzpicture}	
	}
	
	\begin{figure}[h]
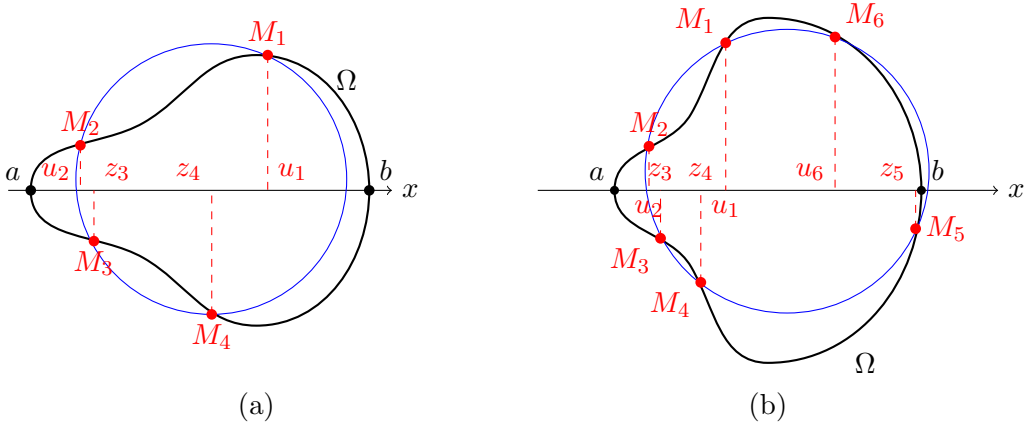

		\centering
		\Mi
		\caption{The intersection points of a minimizing set $\Omega$ and its optimal ball, in the proof of Proposition \ref{centre_optimizing_ball}.}\label{Mi-figure}
	\end{figure}	
	
	Assume by contradiction that the center of the optimal ball is $(0,t), t>0$. 
	We are going to prove that $\frac{\partial \psi}{\partial y}>0$ which leads to a contradiction.
	
	Define 
	$$
	g_1(x)=x^2+(f(x)-t)^2-1; \qquad g_2(x)=x^2+(f(x)+t)^2-1.
	$$
	Notice that the intersection points of the optimal ball with $\Omega$, in $\Pi^+$, satisfy  $g_1(u)=0$, while the intersection points of the optimal ball with $\Omega$ in $\Pi^-$, satisfy  $g_2(z)=0$.
	Observe that $g_2(x)\ge g_1(x)$ for every $x\in\R$ since $g_2(x)-g_1(x)=4f(x)t\ge 0$.  
	Hence there cannot exist zeros of $g_2$ in an interval where $g_1$ is non-negative.
	Moreover for $x\in(a,b)$, $g_2(x)>g_1(x)$, while for $x\le a$ or $x\ge b$, $g_1(x)=g_2(x)=x^2+t^2-1$ and hence $g_1, g_2\to +\infty$ if $x\to \pm \infty$. 
	Let  $u_{i}>u_{i+1}$ be two consecutive zeros of $g_1$ and assume that $g_1<0$ in $(u_{i+1},u_{i})$. 
	Since we assumed $b\ge\sqrt{1-t^2}$, necessarily $i$ is an even index.
	We now focus the analysis of $g_2$ in the interval $(u_{i+1},u_{i})$. 
	
	Three different situations may occur: (i) the function $g_2$ has no zero in the interval $(u_{i+1},u_{i})$; (ii) there exists exactly two zeros of $g_2$ belonging to $(u_{i+1},u_{i})$; (iii) there exist more than four zeros of $g_2$ belonging to $(u_{i+1},u_{i})$.
	We analyse each situation in order to study the sign of the function $x_1+...+x_{2p-1}-(x_2+...+x_{2p})$.
	
	In situation (i), the function $g_2$ is strictly positive, and $u_i-u_{i+1}>0$, that is, $x_i-x_{i+1}>0$.
	In situation (ii), let $z_{l}<z_{l+1}$ be the only two consecutive zeros of $g_2$ belonging to $(u_{i+1},u_{i})$, that is $g_2<0$ in $(z_{l},z_{l+1})$ and positive elsewhere. Hence the index $l$ is even as well as $i$.
	We can say that $u_{i}-u_{i+1}\geq z_{l+1}-z_{l}$ and hence, $x_i+z_l-(x_{i+1}+z_{l+1})>0$. 
	
	In situation (iii), $g_2$ changes its sign more than three times in $(u_{i+1},u_{i})$.
	Let $z_l<z_{l+1}<...<z_{l+q}$ be the zeros of $g_2$; again the indices $l, q$ are even. 
	We can say that $u_i-u_{i+1}>z_{l+1}-z_{l}+...+z_{l+q}-z_{l+q-1}$ which can be rewritten as
	$x_{i}+x_l+...+x_{l+q-1}-(x_{i+1}+ x_{l+1})+...+x_{l+q})>0$. 
	This argument implies that
	$\frac{\partial \psi}{\partial y}>0$, which leads to a contradiction. Hence the center of the optimal ball is at $(0,0)$.
	We have thus proved that all the optimal balls must be centered on $\Pi$ for such domains.
	
	\medskip
	For the general case, we proceed by approximation. 
	Let now $\Omega$ be an arbitrary domain: either if $\Omega$ is not transversal to an optimal ball or if its
	boundary contains vertical segments, one can find  
	a sequence of sets $\Omega_n$ transversal to their optimal ball or without vertical segments
	and converging to $\Omega$ in the $L^1$ norm.
	Let $B_{c_n}$ be a sequence of corresponding optimal balls for $\Omega_n$  of center $c_n$; necessarily $c_n$ belongs to $\Pi$
	according to the first part of the proof. Up to a subsequence $c_n$ converges to some $c\in\Pi$. 
	By definition of optimal ball, 
	\begin{equation}\label{defn_opt_ball}
		|B_{c_n}\Delta \Omega_n|\leq |B_{(x,y)}\Delta \Omega_n|,
	\end{equation} for every $(x,y)\in\R^2$. Now, 
	$$
	|B_{c_n}\Delta \Omega_n|=||{\chi_{B_{c_n}}-\chi_{\Omega_n}}||_{L^1(\R^2)}\to 
	||{\chi_{B_{c}}-\chi_{\Omega}}||_{L^1(\R^2)}.
	$$
	In the same way
	$$
	|B_{(x,y)}\Delta \Omega_n|=||{\chi_{B_{(x,y)}}-\chi_{\Omega_n}}||_{L^1(\R^2)}\to 
	||{\chi_{B_{(x,y)}}-\chi_{\Omega}}||_{L^1(\R^2)}.
	$$
	Therefore, passing to the limit in (\ref{defn_opt_ball}), we obtain that $B_c$ is an optimal ball for $\Omega$. 
\end{proof} 

\medskip
By applying Proposition \ref{centre_optimizing_ball} in two orthogonal directions, we easily deduce the
following corollary. It will be useful to characterize the optimal ball of the symmetrized sets $\Omega^*$
defined in section \ref{sec2}.
\begin{cor}\label{corcentre}
	Assume that  $\Omega\subset \R^2$ is transversal to its optimal balls, it has two perpendicular axes of symmetry $\Pi_1, \Pi_2$,  it is convex in both  perpendicular directions to $\Pi_1, \Pi_2$ and it has no segments on its boundary parallel to one of these directions.  
	Then there exists an optimal ball  centered  at the intersection of the two axes $\Pi_1, \Pi_2$.
	Moreover, if $\partial\Omega$ does not contain segments parallel to directions of symmetry, necessarily it has
	only one optimal ball which is  centered  at the intersection of the two axes.
\end{cor}

We use now the previous corollary and Theorem \ref{casABCD} to compute the infimum of $\liminf_{\varepsilon\to 0}\mathcal{F}({\Omega_{\varepsilon}})
$.
\begin{thm}\label{minimizing_sequences_general}
Let $\varepsilon>0$.
Let $\Omega_{\varepsilon} $ be a sequence of planar regular sets converging to a ball $B$.
Then
$$
\inf
\left\{
\liminf_{\varepsilon\to 0}\mathcal{F}({\Omega_{\varepsilon}})
\right\} 
= 
\frac{\pi}{8(4-\pi)}.
$$
\end{thm}
\begin{proof}
According to Corollary \ref{corfond}, it suffices to exhibit a sequence of domains for which
we have equality in (\ref{mino}). We choose a particular sequence $\Omega_\varepsilon^*$.
If $\Omega_\varepsilon^*$ is convex,  Corollary \ref{corcentre}
guarantees that the center of its optimal ball is the center of symmetry of $\Omega^*$.
Let $\Omega_\varepsilon$
be a sequence of transversal sets to a ball $B$,  converging to $B$ and such that $\Omega_\varepsilon^*$ is convex. This corresponds to the cases $A_1$ and $A_2$ 
 of  the proof of Theorem \ref{casABCD}
for which we have, see (\ref{eq=})
$$
\liminf_{\varepsilon\to 0}\frac{\delta(\Omega_\varepsilon^*)}{\lambda^2(\Omega_\varepsilon^*)}
= \frac{\pi}{8}
\left[\frac{\cos(\hat{\eta}_1)}{8(\sin(\hat{\eta}_1)-\hat{\eta}_1\cos(\hat{\eta}_1))}+
\frac{\cos(\hat{\eta}_2)}{8(\sin(\hat{\eta}_2)-\hat{\eta}_2\cos(\hat{\eta}_2))}\right]\,.
$$
Now the minimal value of the right-hand side is achieved for $\hat{\eta}_1=\hat{\eta}_2=\pi/4$
for which we get
$\frac{\pi}{8(4-\pi)}$,
as seen in (\ref{A1A2}).
\end{proof}

\begin{rem}
Notice that the same result has been proved in \cite{CiLe2} by using the so called ovals sets.
\end{rem}



\section{Existence theorem} \label{sec3}
In this section we are going to prove the existence of a minimizer $\Omega$ for the functional $\mathcal{F}$.
We have proved that  sequences of sets converging to a ball cannot be minimizing.
For the other sequences, 
we are going to prove that they are  contained in a fixed bounded domain $R$. This will allow us to get the existence of a limit set $\Omega$. 
More precisely in Proposition \ref{perimeter} below we prove a uniform bound on the number of connected components of a minimizing sequence. 
Notice that a similar uniform boundedness was proved in
\cite[Lemma 5.1]{FMP} in a different way.

We start with a natural result in this context.
\begin{lemma}\label{comp_conn_loin}
Let $\Omega=\cup_i \omega_i\subset \R^2$, where $\omega_i$ are connected components. Let $\widetilde{\Omega}$  be any set composed by the same $\omega_i$ translated in such a way that the Euclidean distance between them is more than one. Then $\lambda(\Omega)\leq \lambda(\widetilde{\Omega})$.
\end{lemma}
\begin{proof}
We will prove the statement for a set $\Omega$ composed by two connected components $E_1$ and $E_2$; the general case is similar.

Assume $|\Omega|=\pi$ and let $B$ be an optimal ball for $\Omega$, that is, $\la(E_1\cup E_2)=2(\pi-|E_1\cap B|-|E_2\cap B|)/\pi$. 
We denote by $\widetilde{\Omega}$ the set obtained by translating far away the component $E_2$. 
Up to rename $E_1, E_2$, we can assume that  $\la(\widetilde{\Omega})=2(\pi-|B_x\cap E_1|)/\pi$, for some ball $B_x\neq B$ of radius one and center at $x$. 
Assume by contradiction that $\la(\Omega)>\la(\widetilde{\Omega})$, that is, $|E_1\cap B|+|E_2\cap B|<|E_1\cap B_x|$. Hence $|E_1\cap B|+|E_2\cap B|<|E_1\cap B_x|+|E_2\cap B_x|$, which contradicts the fact that $B$ is an optimal ball for the Fraenkel asymmetry of $\Omega$.
\end{proof}

\begin{lemma}\label{lemmaperimeter}
Let $\omega\subset \R^2$ be a connected set which is not contained in a ball of radius $1$. Then its perimeter is greater than 4.
\end{lemma}
\begin{proof}
The convex hull of $\omega$, denoted by $co(\omega)$, is connected and is not contained in a ball of radius $1$. Hence the circumradius of $co(\omega)$ is greater than 1.
Therefore the perimeter of $co(\omega)$ is greater or equal to 4. 
Since $\omega\subset \R^2$ is connected,  its perimeter is greater or equal the perimeter of its convex hull. Thus, the perimeter of $\omega$ is greater than 4.
\end{proof}

\begin{prop}\label{perimeter}
Let $\Omega$ be a planar set whose perimeter is less than 20. Then there exists a planar set  $\widetilde{\Omega}$ composed by at most 7 connected components, such that 
$$
\mathcal{F}(\widetilde{\Omega})\le \mathcal{F}(\Omega).
$$
\end{prop}

\begin{proof}[Proof of Proposition \ref{perimeter}]
Assume that $\Omega$ is the union of $m$ (possible infinite) connected components $\omega_i$.  
By Lemma \ref{comp_conn_loin}, we can assume that only one connected component $\omega_k$ has a non-empty intersection with an optimal ball $B$, since one can translate far away each connected component and this procedure decreases the value of the functional $\mathcal{F}$ (since $\lambda$ increases and $\delta$ keeps equal). 

By Lemma \ref{lemmaperimeter}, there exist at most 4 connected components $\omega_1, \omega_2, \omega_3, \omega_4$ which are not contained in a ball $B_1$ of radius 1. 
The first
step consists in replacing all the other components $\omega_i, i\geq 5$ by a ball: this decreases the perimeter without changing the Fraenkel asymmetry. We  relabel all these balls by choosing a decreasing order with respect to the corresponding radii $r_0 \geq r_1\geq r_2...$.

The optimal ball is  either on one of the four first components $\omega_j, j=1,\ldots 4$ or on the
ball with the largest radius $\omega_5=B_{r_0}$.  
So we deal  with a domain $\widehat{\Omega}$ defined as
$$
\widehat{\Omega}=\bigcup_{j=1}^{4}\omega_j\cup \bigcup_{i\ge 0}B_{r_i},
$$ 
for which $\mathcal{F}(\widehat{\Omega})\leq \mathcal{F}(\Omega)$.
 Moreover $\lambda(\widehat{\Omega}=\widehat{\lambda}(\widehat{\Omega})$, where
$$
\widehat{\lambda}(\widehat{\Omega})=
\min\left\{ 2(1-r_0^2);\ \frac{|\omega_j\Delta B_1|+\pi-|\omega_j|}{\pi},\ j=1,...,4
\right\}\,.
$$  
Notice that, since $\widehat{\lambda}(\widehat{\Omega})\leq 2(1-r_i^2)$, it holds 
$r_i\leq a:=\sqrt{1-\widehat{\lambda}(\widehat{\Omega})}$, for every $i\geq 1$.

Let  $\pi\alpha$ be the area of $\omega_1\cup \omega_2\cup \omega_3\cup \omega_4\cup B_{r_0}$ and $2\pi P_0$ its perimeter.
Let  $K$ be the following compact set:
$$
K:=\left\{\underline{r}=(r_1, r_2,...)\ :\ \sum_{i\geq 1} r_i^2=1-\alpha,\ 0\leq r_i\leq a\right\}.
$$
We construct the set $\widetilde{\Omega}$ by replacing the balls $B_{r_1}, B_{r_2},...$ of $\Omega_0$ by two balls, at most,  in such a way that  $\frac{\delta}{\lambda^2}$ decreases. To do that, we minimize over $K$ the quantity 
$$
j(\underline{r})=\frac{P_0-1+\Sigma_{i\geq 1} r_i}{\widehat{\lambda}^2(\Omega_0)}.
$$
The minimizer has at most one element $r_k$ verifying $0<r_k<a.$ Indeed, if there exists $r_i\neq r_j$ such that $0<r_i<a,  0< r_j<a$, then we can replace $r_i$ by $r_i+h_1$ and $r_j$ by $r_j-h_2$ for some $h_1<h_2$, in such a way that the strict inequalities are still satisfied. We still have an element of $K$, because 
$$
1-\alpha=\sum_{k\neq i, k\neq j} r_k^2+(r_i+h_1)^2+(r_j-h_2)^2\,,
$$
if and only if $(h_1, h_2)$ belongs to the cercle centered in $(-r_i,r_j)$ of radius $\sqrt{r_i^2+r_j^2}$.
It is easy to see that such $(h_1, h_2)$ exist: this implies that
$r_i+h_1$ and $r_j-h_2$  give a smaller value of the minimum of
the quantity $j(\underline{r})$. This is absurd.
Hence the minimum of $j(r)$ over $K$ is attained for $\underline{r}=(a, a, a, ..., a, b, 0, 0, ...)$ where $a$ is repeated $m$ times and $b<a$. 
The minimizer satisfies $|\underline{r}|^2=m a^2+b^2=1-\alpha$, that is, $m\leq \frac{1-\alpha}{a^2}$. 
Notice that the quantity $j(\underline{r})$ is minimal for the values of $a,m$ which minimize $a(m+\sqrt{\frac{1-\alpha}{a^2}-m})$. 
It is not difficult to prove that the minimum is realized by $m=1$. 
Therefore, the set $\widetilde{\Omega}=\omega_1\cup \omega_2\cup \omega_3\cup \omega_4\cup B_{r_0}\cup B_{r_1}\cup B_{r_2}$ satisfies $\mathcal{F}(\widetilde{\Omega})\le\mathcal{F}(\Omega)$.
\end{proof}

We are now able to prove Theorem \ref{thm_existence_minimizer}.
\begin{proof}[Proof of Theorem \ref{thm_existence_minimizer}]
Let $\{\Omega_n\}$ be a minimizing sequence for $\mathcal{F}$ and assume $|\Omega_n|=\pi$. 
By Corollary \ref{corfond} and Theorem \ref{thmAFN} the sequence $\{\Omega_n\}$ does not converge to a ball.
Since $\delta(\Omega_n)\leq 0.41 \lambda^2(\Omega_n)$ and  $\lambda(E)\leq 2$, for any planar set $E$,it holds $\delta(\Omega_n)\leq 1.64$. 
This implies that $P(\Omega_n)\leq 2\pi+2\pi\;1.64 \leq 16.6$. 
 
By  Proposition \ref{perimeter} we can replace $\Omega_n$ by another minimizing sequence (still denoted $\Omega_n$)
with at most 7 connected components and each component has a diameter less or equal to 8.4 (because the total perimeter
is less than 16.6). Therefore, it is possible to enclose all the connected components, with mutual distance between 1 and 2
in a large, but fixed, rectangle $R$. 

By recalling that $P(\Omega_n)=|D\chi_{\Omega_n}(R)|$, the sequence $\chi_{\Omega_n}$ is bounded in $BV(R)$.
By the compact embedding $BV(R)\hookrightarrow L^1(R)$,  
there exists $\Omega\subset R$ such that, up to a subsequence,
$\chi_{\Omega_n}\to \chi_{\Omega}$ weakly-$*$ in $BV(R)$. 
Thus $\chi_{\Omega_n}\to \chi_{\Omega}$ in $L^1(R)$ and $|\Omega|=\pi$. 
Notice that $\lambda(\Omega)>0$ since  $\Omega$ is not a ball.
Moreover
$\liminf\limits_{n\to \infty} P(\Omega_n)\geq P(\Omega)$
by lower semicontinuity and thus $\liminf\limits_{n\to \infty} \delta(\Omega_n)\geq \delta(\Omega)$.

We finally prove that
$\lambda(\Omega_n)\to \lambda(\Omega)$.
Let $B_x, B_{x_n}$ be optimal balls, with respect to the  Fraenkel asymmetry, of  $\Omega, \Omega_n$, respectively and let $x,x_n$ be their centers. 
Therefore $|\Omega_n|=|B_{x_n}|=|\Omega|=|B_x|=\pi$ and $P(B_x)=P(B_{x_n})=2\pi$.
It holds
\begin{eqnarray*}
\lambda(\Omega_n)&\leq& \frac{|\Omega_n \Delta B_x|}{|B_x|}=\frac{||\chi_{\Omega_n}-\chi_{B_x}||_{L^1(R)}}{\pi}\\
&\le& 
\frac{||\chi_{\Omega_n}-\chi_{\Omega}||_{L^1(R)}}{\pi}+\frac{||\chi_{\Omega}-\chi_{B_x}||_{L^1(R)}}{\pi}=\ep_n+\la(\Omega).
\end{eqnarray*}
On the other hand, 
$$
\lambda(\Omega)\leq \frac{|\Omega \Delta B_{x_n}|}{|B_{x_n}|}\leq 
\frac{||\chi_{\Omega}-\chi_{\Omega_n}||_{L^1(R)}}{\pi}+\frac{||\chi_{\Omega_n}-\chi_{B_{x_n}}||_{L^1(R)}}{\pi}=\ep_n+\lambda(\Omega_n),
$$ 
and hence we have
$$
\lim_{n\to+\infty}\lambda(\Omega_n)= \lambda(\Omega)\,.
$$ 
Thus $\Omega$ provides a solution of the minimization problem.
\end{proof}


\section{Properties of the optimal set}\label{sec4}
In this section we gather some analytic and geometric properties of an optimal set for $\mathcal{F}$. 
\begin{prop}\label{thmqualit}
	Let $\Omega_0$  be a minimizer of the functional $\mathcal{F}$. Then, 
	\begin{enumerate}
		\item $\partial \Omega_0$ is of class $C^{1,1}$;
		\item the boundary of $\Omega_0$ is composed of arcs of circle. More precisely, 
		in any connected component  of the set $\R^2\setminus \cup_{x\in {Z}(\Omega_0)} (x+\partial B)$
		(where $Z(\Omega_0)$ is the set of the centers of the optimal balls for $\Omega_0$),
		$\partial \Omega_0$ is an union of arcs of circle with the same radius;
		\item $P(\Omega_0)\le 16.16$;
		\item $\Omega_0$ is not convex;
				\item $\Omega_0$ is composed by  at most 6 connected components.
	\item\label{4}  $\Omega_0$ has at least two optimal balls realizing the Fraenkel asymmetry;
	\end{enumerate}
\end{prop}
The first two statements have been proved in \cite{Ci-Le}.  
Statement (3) 
follows from the fact that $\lambda(\Omega)\le 2$ for every $\Omega$ and $\mathcal{F}(\Omega_0)\le 0.41$ by Theorem \ref{thmAFN}, so that $P(\Omega_0)\le 2\pi(4\cdot 0.41-1)=16.16$.
 Statement (4) 
 follows from the existence
of a non convex set $\textsf{M}_0$, 
 shown in Section \ref{section-conjecture}, for which $\mathcal{F}({\textsf{M}}_0)\approx 0.39$ (see Conjecture \ref{conj2}).

We sketch the proof of statement (5). 
Assume that the optimal domain
has several connected components: $\Omega_0=\cup_{i=1}^m \omega_i$ ($m$ possibly infinite) and assume $|\Omega_0|= \pi$.
Necessarily, if a component $\omega$ is not contained in a unit ball, then $|\omega\Delta B_x|>0$, as $B_x$ is an optimal ball for the Fraenkel
asymmetry; otherwise, we could replace it by a ball strictly decreasing the perimeter. Therefore, we can use an analogous argument to that one 
of the proof of Proposition \ref{perimeter}, noticing moreover that the optimal ball is positioned on one of the four first components $\omega_i, i=1,\ldots 4$
(and actually on all) and that $P(\Omega_0)<20$ by statement  (3). Thus we can perform the minimization procedure shown in the proof of Proposition \ref{perimeter}
and we are able to replace the collection of balls by at most two balls (the biggest one which could also be
in contact with an optimal ball).

Property (\ref{4}) is proved in Section \ref{subsection_number_optimal_balls}, by using some optimality conditions satisfied by an optimal set  for $\mathcal{F}$.
We will compute the shape derivative of the Fraenkel asymmetry to prove these conditions. 
\subsection{Differentiability of the Fraenkel asymmetry}\label{sec5}
We are going to compute the shape derivative of the functional $\Omega\mapsto \la(\Omega)$. 
Since $\la$ is defined as a minimum, we first present a general lemma on differentiability 
of such functional in topological spaces.

Let $A,B$ be two topological spaces. We consider a function $j(x,\zeta):A\times B\to\R$ and we assume that the derivative of
$j$ with respect to the second variable $\zeta$ exists and is continuous with respect to $x$.
For each fixed $\zeta$, we define $\widehat{x}(\zeta)$ as a solution of $\min_{x\in A} j(x,\zeta)$. Let  $\lambda$ be defined on $B$ as the value of
the minimum: $\lambda(\zeta)=j(\widehat{x}(\zeta),\zeta)$.
\begin{lemma}\label{j-diff}
	 Assume that for some $\zeta_0\in B$ there exists a unique $\widehat{x}(\zeta_0)\in A$  where $j(\cdot,\zeta_0)$ attains its minimum.
	Then the function $\zeta \mapsto j(\widehat{x}(\zeta),\zeta)$ is differentiable at $\zeta_0$ and its derivative 
	is $\frac{\partial j}{\partial \zeta}(\widehat{x}(\zeta_0),\zeta_0)$.
\end{lemma}
Even if the previous result is classical in variational analysis, we prove it for sake of completeness.
\begin{proof}
Fix $\zeta_0,h\in B$ and let $t>0$ (the proof works as well for $t<0$). 
We are going to compute the derivative of $\la(\zeta)=j(\widehat{x}(\zeta),\zeta)$ along the 
direction $h$.
By definition of $\widehat{x}(\zeta)$ it holds 
$$
\lambda(\zeta_0+th)=j(\widehat{x}(\zeta_0+th),\zeta_0+th)\le j(\widehat{x}(\zeta_0),\zeta_0+th)\,.
$$
This implies that
$$
\frac{\lambda(\zeta_0+th)-\lambda(\zeta_0)}t \le  \frac{j(\widehat{x}(\zeta_0),\zeta_0+th)-j(\widehat{x}(\zeta_0),\zeta_0)}t
$$
and hence, passing to the limit as $t\to 0$, we get
$$
\limsup_{t\to 0} \frac{\lambda(\zeta_0+th)-\lambda(\zeta_0)}t \le \langle\frac{\partial j}{\partial \zeta}
(\widehat{x}(\zeta_0),\zeta_0);h\rangle.$$
The reverse inequality is obtained in an analogous way by observing that
$$
\lambda(\zeta_0)=j(\widehat{x}(\zeta_0),\zeta_0)\le  j(\widehat{x}(\zeta_0+th),\zeta_0),
$$
which implies
$$
\frac{\lambda(\zeta_0+th)-\lambda(\zeta_0)}t \ge  \frac{j(\widehat{x}(\zeta_0+th),\zeta_0+th)-j(\widehat{x}(\zeta_0+th),\zeta_0)}t.
$$
By the uniqueness of $\widehat{x}(\zeta_0)$, we have $\widehat{x}(\zeta_0+th) \to \widehat{x}(\zeta_0)$ as $h\to 0$. The continuity of the function 
$x \mapsto  \frac{\partial j}{\partial \zeta}(x,\zeta)$ implies
$$
\liminf_{t\to 0} \frac{\lambda(\zeta_0+th)-\lambda(\zeta_0)}t \ge \langle\frac{\partial j}{\partial \zeta}
(\widehat{x}(\zeta_0),\zeta_0);h\rangle,
$$
which gives the desired result.
\end{proof}
We are going to apply the previous lemma to $j(X,\Omega):=|B_X\Delta \Omega|/\pi$ in such a way that $\widehat{X}(\Omega)$
is the center of an optimal ball and $\lambda(\Omega)=j(\widehat{X}(\Omega),\Omega)$ is the Fraenkel asymmetry.
Notice that in this way we make an abuse of language, since $\Omega$ in fact does not belong to a topological space. Instead
we should consider the classical shape derivative, as explained in \cite[Chapter 5]{HP} and replace the set  $\Omega$
by a space of diffeomorphisms acting on a fixed domain. Since no confusion can occur, we keep this convenient way to present
the derivative. In that context, what we denote by $d\lambda(\Omega;V)$ is the limit, as $t\to 0$, of the ratio
$$
\frac{\lambda((Id+tV)(\Omega))-\lambda(\Omega)}t\,,
$$ 
where $V:\R^2\to\R^2$ 
is any regular vector field and $Id+tV$ a small
perturbation of the identity operator.
\begin{prop}\label{shdlambda}
Let $\Omega$ be a planar regular (Lipschitz) set of area $\pi$ and assume that it has a unique optimal ball for the Fraenkel asymmetry $\la$,
whose center is at $X^*$. Denote by $\Omega^{OUT}=\Omega\setminus B_{X^*}$ and $\Omega^{IN}=B_{X^*}\setminus \Omega$.
Then the shape derivative of $\la$ exists and is given by
\begin{equation}\label{der1}
d\la (\Omega;V)=\int_{\partial\Omega^{OUT}\setminus\partial B_{X^*}} V\cdot n \,ds - 
\int_{\partial\Omega^{IN}\setminus\partial B_{X^*}} V\cdot n \,ds,
\end{equation}
where $n$ is the exterior normal unit vector to the boundary of $\Omega$.
\end{prop}
\begin{proof}
Assumptions of Lemma \ref{j-diff} are satisfied since the measure is differentiable, see for example
\cite[Theorem 5.2.2]{HP}. 
Notice that the opposite sign of the two terms in  (\ref{der1}) is due to the fact that the measure
of $\Omega$ is counted positively outside $B_{X^*}$ and negatively inside. Moreover, it is clear that the derivative
given in formula (\ref{der1}) is continuous with respect to $X$.
\end{proof}

\subsection{The optimal set has at least two optimal balls}\label{subsection_number_optimal_balls}
In this section we are going to prove that an optimal set for $\mathcal{F}$ cannot have only one optimal ball for the Fraenkel asymmetry. 
\begin{thm}
	Let $\Omega$ be a minimizing domain for the functional $\mathcal{F}$. Then $\Omega$ has at least two optimal balls for the Fraenkel asymmetry.
\end{thm}
The proof argues by contradiction. Notice that if $\Omega$, whose area equals $\pi$, has just one optimal ball $B$, then its boundary is composed by arcs of circles of radius $R_0$ outside $B$ and of radius $R_1$ inside $B$. 
Since $\partial \Omega$ is of class $C^1$,  each connected component $\omega$ of $\Omega$ has a rotational symmetry of order $N\geq 2$.
Let $x,y$ be two points on the boundary of $\omega$ such that the angle $\widehat{xOy}=\frac{\pi}{N}$ and the part of $\partial\omega$ between $x,y$ generates the whole boundary of $\omega$.
Since $\mathcal{F}$ is rotational invariant, we can assume the point $x$ to belong to the horizontal axis.
The strategy consists in using a simple parametrization of the boundary of $\Omega$ with two angles $\alpha,\theta$
as described below and to eliminate all possible values of these parameters by contradicting either a first order
or second order optimality condition or proving that the value of the functional $\mathcal{F}$ is greater than
0.406 (see Propositions \ref{optimalconditions_connected} and \ref{optimalconditions_NONconnected}).

We will distinguish the case where $\Omega$ is connected from the case where $\Omega$ is not connected in the next two subsections.
Figure \ref{fig-intervalli} summerizes the structure of the proof in the two cases, in terms of the range of the parameters $\alpha,\theta$.

\newcommand{\intervalli}{
	\begin{tikzpicture}[x=5mm,y=5mm]
	\draw[->] (0,0)node[below right]{$0$}--(10,0)node[right]{$\theta$};\draw[->] (0,0)--(0,11)node[left]{$\alpha$};
	\draw (0,3)node[ left]{$\dfrac{\pi}{N}$}--(8,3);
	\draw (4,0)node[below]{$\dfrac{\pi}{2N}$}--(4,0.2);\draw(4,3)--(4,10);\draw (0,10)node[left]{$\pi$}--(8,10)--(8,3);\draw[dotted](8,3)--(8,0)node[below]{$\dfrac{\pi}{N}$};
	\draw (0,5)node[left]{$\dfrac{\pi}{2}$}--(4,5);
	\draw (0,7.5)node[left]{$\dfrac{\pi}{2}+\dfrac{\pi}{N}$}--(4,7.5);
	\draw (0,9.5)[right]node{\footnotesize{Lemma \ref{48}}};\draw (0,9)node[right]{\footnotesize{Lemma \ref{49}}};\draw (0,8.5)node[right]{\footnotesize{Lemma \ref{411}}};
	\draw (0,6)node[right]{\footnotesize{Lemma \ref{47}}};
	\draw (0,4)node[right]{\footnotesize{Lemma \ref{45}}};
	\draw (4,7)node[right]{\footnotesize{Lemma \ref{44}}};
	\draw (4,-2.5)node{(a)};
	\begin{scope}[xshift=7cm]
	\draw[->] (0,0)node[below right]{$0$}--(10,0)node[right]{$\theta$};\draw[->] (0,0)--(0,11)node[left]{$\alpha$};
	\draw (4,1.5)--(4,10);\draw (4,0)node[below]{$\dfrac{\pi}{2N}$}--(4,0.2);
	\draw (0,10)node[left]{$\pi$}--(8,10)--(8,3);\draw[dotted] (8,3)--(8,0)node[below]{$\dfrac{\pi}{N}$};
	\draw (0,7)node[left]{$\alpha(N)$}--(4,7);
	\draw (0,3)node[left]{$\dfrac{\pi}{N}$}--(4,3);
	\draw (0,8)node[right]{\footnotesize{Lemma \ref{415}}};
	\draw (0,5)node[right]{\footnotesize{Lemma \ref{417}}};
	\draw (4,8)node[right]{\footnotesize{Lemma \ref{416}}};
	\draw (0,2)node[right]{\footnotesize{Lemma \ref{418}}};
	\draw (4,-2.5)node{(b)};
	\draw (0,0)--(8,3); \path (6,2.5)node[below,rotate=20]{$\alpha=\theta$};
	\end{scope}
	\end{tikzpicture}
}
\begin{figure}[h]
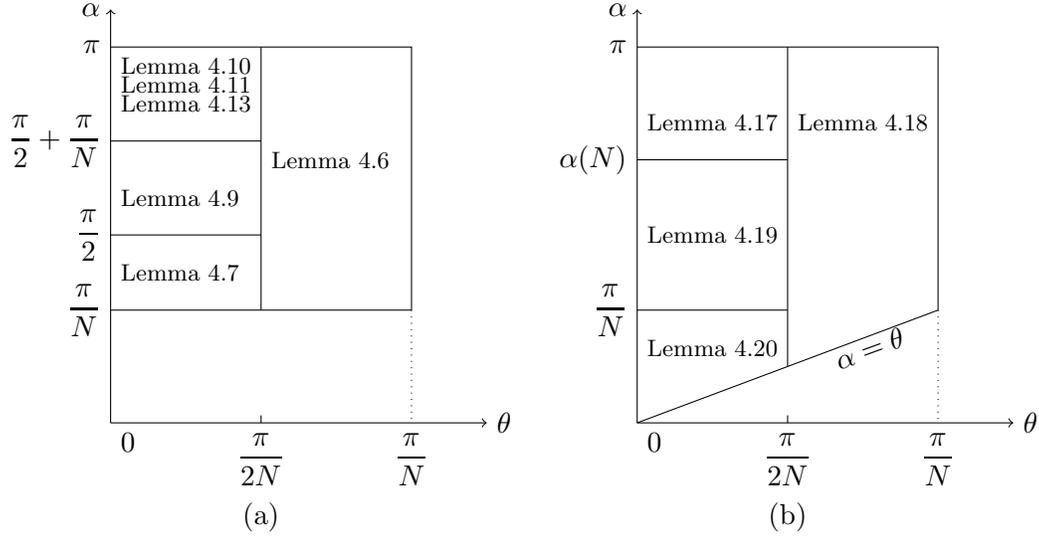

	\centering
	\intervalli
	\caption{(a) Connected case, $N\ge 4$. (b) Non-connected case, $N\ge 3$.}\label{fig-intervalli}
\end{figure}

\subsubsection{Connected case}
First of all, let us remark that, by statement (4) of Proposition \ref{thmqualit}, we can restrict the analysis to non-convex domains.
Let $M$ be an intersection point between $\partial \Omega$ and $\partial B$, $A$ be the center of the arc $\Gamma_0$ of radius $R_0$ and 
$B$ be the center of the arc $\Gamma_1$ of radius $R_1$. Since $\Omega$ is of class $C^1$, the points $A, M, B$ belong to a straight line. 

\newcommand{\figuraCasoConnesso}{
\begin{tikzpicture}[x=4mm,y=4mm] 
\fill[gray!20] (3.75,0)--(4.75,0)arc(0:74:1);\draw (4.5,0.6)node{$\alpha$};
\fill[gray!50] (0,0)--(1,0)arc(0:28.6:1);\draw (1,0.5)node[right]{$\theta$};
\fill[gray!40] (60:14)--(60:13)arc(240:254:1);\draw (55:14)node[below]{$(\alpha-\frac{\pi}{N})$};
	\draw[thick] (0,0) circle(5);
	\draw[thick, blue] (6.25,0)arc(0:75:2.5); \draw[thick,dashed,blue] (6.25,0)arc(0:-75:2.5);
	\draw[thick,dashed, blue] (91.3:5)arc(225:240:10.3);\draw[thick,blue](60:4)arc(240:254:10.3); 
	\draw[thick, dashed, blue] (91.3:5)arc(45:195:2.5);
	\draw[thick,dashed,blue] (148.7:5)arc(15:-14:10.3);
	\draw[thick,dashed,blue] (211.3:5)arc(166:314:2.5);
	\draw[thick,dashed,blue] (-91.3:5)arc(136:106:10.1);
	\draw (3.75,0)--(28.7:5); \draw (3.75,0)--(7,12.18);
	\draw (0,0)--(28.7:5); 
	\draw[red] (0,0)--(8,0);\draw[red](0,0)--(60:14);\draw[red] (0,0)node[below, left]{$O$};
	\draw (28.7:5.2)node[right]{$M$};\fill (28.7:5)circle(1.5pt);\draw (3.7,0)node[below]{$A$};
	\draw[blue] (6.25,1)node[above, right]{$\Gamma_0$};
	\draw (60:14.2)node[right]{$B$};
	\draw[blue] (30:4)node[below,left]{$\Gamma_1$};
	\draw (6,0)node[below right]{$x$};\draw (60:3.7)node[left]{$y$};\fill (60:4)circle(1.5pt); \fill (6.25,0)circle(1.5pt);
	\fill[red] (0,0)--(0.5,0)arc(0:60:0.5);\draw[red,<-] (.5,-0)--(0.5,-0.5)node[below]{$\frac{\pi}{N}$};
	\end{tikzpicture}}
\begin{figure}[h]
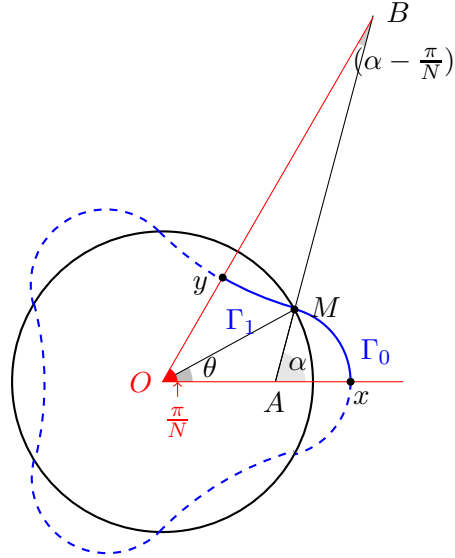

	\centering
	\figuraCasoConnesso
	\caption{Connected case. A set $\Omega$ with one optimal ball, $N=3$.}
\end{figure}

Let $\theta\in (0, \frac{\pi}{N})$ be the angle between $Ox$ and $OM$ and $\alpha\in (\frac{\pi}{N},\pi)$ be the angle between $Ax$ and $AM$. Notice that, by construction,  $0\le\theta\le\frac{\pi}{N}\le\alpha\le\pi$. The parametrization of the straight line passing through $A, M, B$ is
$x(t)=\cos \theta +t \cos \alpha,\ y(t)=\sin \theta +t \sin \alpha$, for $t\in \R$. Therefore the point $A$ can be obtained for $t_A=-\frac{\sin(\theta)}{\sin(\alpha)}$
and $B$ for
$t_B=-\frac{\sin(\pi/N-\theta)}{\sin(\alpha-\pi/N)}$. Therefore
\begin{equation}\label{R0R1}
R_0=|t_A|=\frac{\sin(\theta)}{\sin(\alpha)}, \qquad
R_1=|t_B|=\frac{\sin(\frac{\pi}{N}-\theta)}{\sin(\alpha-\frac{\pi}{N})}\,,
\end{equation}
and so
$$
|\Gamma_0|=\alpha \frac{\sin(\theta)}{\sin(\alpha)}, \qquad |\Gamma_1|=\left(\alpha-\frac{\pi}{N}\right) \frac{\sin(\frac{\pi}{N}-\theta)}{\sin(\alpha - \frac{\pi}{N})}.
$$ 
The perimeter of $\Omega$ equals $2N(|\Gamma_0|+|\Gamma_1|)$. Therefore
$$
\delta(\Omega)=
\frac{2N\left[\alpha \frac{\sin \theta}{\sin \alpha}+ \left(\alpha-\frac\pi N\right)\frac{\sin\left(\frac{\pi}{N}-\theta\right)}{\sin(\alpha - \frac{\pi}{N})}\right]-2\pi}{2\pi}.
$$ 
We are going to compute $\lambda(\Omega)=\frac{N}{\pi}(\mathcal{A}_0 +\mathcal{A}_1)$, where $N \mathcal{A}_0$ is the area of $\Omega\setminus B$ and 
$N \mathcal{A}_1$ is the area of $B \setminus \Omega$.
Let $g$ be defined by (\ref{defn_fonctions_gh}); notice that  
\begin{equation}\label{A0A1}
\mathcal{A}_0=R_0^2 g(\alpha) - g(\theta), \qquad \mathcal{A}_1=R_1^2 g\left(\alpha-\frac\pi N\right) + g\left(\frac \pi N-\theta\right).
\end{equation}
The proof will be splitted into several parts, according to the values $\theta, \alpha, N$. For every part, the contradiction will be given by the fact that $\Omega$ does not satisfy one of the  conditions 
expressed in the next proposition.
\begin{prop}\label{optimalconditions_connected}
	Let $\Omega$ be a connected planar set which minimizes the functional $\mathcal{F}$.
	Assume that $\Omega$ has a unique optimal ball and let $N\ge 2$ be the order of its rotational symmetry. 
	Then the following conditions hold.
	\begin{enumerate}
		\item[(i)]
		$\mathcal{A}_0=\mathcal{A}_1\leq \frac{\pi}{N}$, where $\mathcal{A}_0$ and $\mathcal{A}_1$ are defined in {\rm(\ref{A0A1})}.
		\item[(ii)]
		$
		\dfrac{1}{R_0}+\dfrac{1}{R_1}=\dfrac{8\delta(\Omega)}{\lambda(\Omega)}
		$, where $R_0$ and $R_1$ are defined in {\rm(\ref{R0R1})}.
		\item[(iii)]
		$\mathcal{F}(\Omega)< 0.406$.
		\item[(iv)] 
		$
		Q\geq 0
		$
		where $Q:=\frac{1}{\sin^3(\theta)}H(\alpha)-\frac{1}{\sin^3(\frac{\pi}{N}-\theta)}H(\alpha-\frac{\pi}{N})-\frac{32 N}{\pi}\frac{\delta(\Omega)}{\lambda^2(\Omega)}$, with  $H(x)=\frac{\sin^3(x)}{\tan(x)-x}$.
	\end{enumerate}
\end{prop}
\begin{proof}
	$(i)$	
	    The equality follows from the fact that $\pi=|\Omega|=|\Omega\setminus B|+|B\setminus \Omega|=|B\setminus \Omega|+|B\setminus \Omega|=|B|$.

	$(ii)$
		This condition comes from the first order optimality condition.
		Indeed, the general optimality condition is $d(\frac{\delta}{\lambda^2},V)=\mu\ d(Area,V)$, where $\mu$ is a Lagrange multiplier due to the fact that  the areas are fixed, and 
		$V: \R^2\to \R^2$ is any regular vector field. This gives
		\begin{eqnarray*}
		&\displaystyle \frac{1}{\lambda^2}\left[\frac{1}{2\pi}\int_{\Gamma_0}\frac{1}{R_0}V\cdot n-\frac{1}{2\pi}\int_{\Gamma_1}\frac{1}{R_1}V\cdot n\right]-\frac{2}{\pi}
		\frac{\delta}{\lambda^3}\left[\int_{\Gamma_0}V\cdot n-\int_{\Gamma_1}V\cdot n\right]&\\
		&\displaystyle =\mu\left[\int_{\Gamma_0}V\cdot n+\int_{\Gamma_1}V\cdot n\right].&
			\end{eqnarray*}
		Since this is true for every $V$, we obtain 
		$$
		\frac{1}{R_0}=4\frac{\delta}{\lambda}+2\pi\mu \lambda^2,\qquad \frac{1}{R_1}=4\frac{\delta}{\lambda}-2\pi\mu \lambda^2,
		$$
		and hence $\frac1{R_0}+\frac1{R_1}=\frac{8\delta}{\la}$.
	
	 $(iii)$
This is a consequence of Theorem \ref{thmAFN}.

	 $(iv)$
This is actually a second order optimality condition.		We are going to modify $\Omega$ by replacing $\alpha$ by $\alpha + \varepsilon_0$ in $\Omega\setminus B_1$ and
		by replacing $\alpha$ by $\alpha - \varepsilon_1$ in $B_1\setminus \Omega$ in such a way that the area of $\Omega$ 
		is preserved. 
		We have
		$$
		R_0^{\varepsilon_0}=
		\frac{\sin \theta}{\sin \alpha}+\varepsilon_0 \cos \alpha-\frac{\varepsilon_0^2}{\sin \alpha}=R_0[1-\varepsilon_0 \cot \alpha +\varepsilon_0^2(1/2+\cot^2 \alpha)],
		$$
		and
		$$
		(R_0^{\varepsilon_0})^2=R_0^2[1-2 \varepsilon_0 \cot \alpha +\varepsilon_0^2(1+3\cot^2 \alpha)].
		$$
		Moreover 
		$
		\mathcal{A}_0^{\varepsilon_0}=(R_0^{\varepsilon_0})^2g(\alpha+\varepsilon_0)-g(\theta)
		$
		and
		$
		g(\alpha+\varepsilon_0)=g(\alpha)+\varepsilon_0 (1-\cos(2\alpha))+\varepsilon_0^2\sin(2\alpha)
		$.
		Therefore 
		$$
		\mathcal{A}_0^{\varepsilon_0}=\mathcal{A}_0+\varepsilon_0 R_0^2 2(1-\alpha \cot \alpha)+\varepsilon_0^2 R_0^2 (\alpha+3\alpha \cot^2\alpha -3\cot \alpha)\,.
		$$
		By using an anoalogous argument we obtain
		$$
		\mathcal{A}_1^{\varepsilon_1}=\mathcal{A}_1+\varepsilon_1 R_1^2 2(1-\beta \cot \beta)-\varepsilon_1^2 R_1^2 (\beta+3\beta \cot^2\beta -3\cot \beta),
		$$
		where 
		$\beta=\alpha-\frac{\pi}{N}$.
		Keeping the total area constant, it holds
		\begin{eqnarray*}
	    &2R_0^2(1-\alpha \cot \alpha)\varepsilon_0 + R_0^2\varepsilon_0^2[\alpha+3\alpha \cot^2\alpha -3\cot \alpha]&\\
	    &=
		2R_1^2(1-\beta \cot \beta)\varepsilon_1 - R_1^2\varepsilon_1^2[\beta+3\beta \cot^2\beta -3\cot \beta].&
		\end{eqnarray*}
		Notice that we can express $\varepsilon_1$ as a function of $\varepsilon_0$, in the form $\varepsilon_1=a\varepsilon_0+b \varepsilon_0^2$, with 
		\begin{eqnarray*}
		&&a=\frac{R_0^2}{R_1^2}\frac{1-\alpha \cot \alpha}{1-\beta \cot \beta},\\
		&&b=\frac{R_0^2(\alpha+3\alpha \cot^2\alpha -3\cot \alpha)+R_1^2(\beta+3\beta \cot^2\beta -3\cot \beta)a^2}{2R_1^2(1-\beta\cot\beta)}.
		\end{eqnarray*}
		Let us consider the variations of the perimeter, that is, $\Delta P=\Delta P_{e}+\Delta P_i$, where 
		\begin{eqnarray*}
		\Delta P_{e}&=&2R_0^{\varepsilon_0}(\alpha+\varepsilon)-2R_0\alpha\\
		&=&2R_0\varepsilon_0(1-\alpha \cot\alpha)+R_0\varepsilon_0^2(\alpha+2\alpha\cot^2\alpha-2\cot\alpha), \\
		\Delta P_{i}&=&
		2 R_1^\varepsilon(\beta-\varepsilon_1)-2R_1\beta\\
		&=&2R_1\varepsilon_1(1-\beta \cot \beta)-R_1\varepsilon_1^2(\beta+2\beta\cot^2\beta-2\cot\beta).
			\end{eqnarray*}
	  Recalling that $\varepsilon_1=a\varepsilon_0+b \varepsilon_0^2$, we get
		$$
		\Delta P_{i}=
		2\frac{R_0^2}{R_1}(1-\alpha\cot\alpha)\varepsilon_0 +
		\left[\frac{R_0^2}{R_1}(\alpha+3\alpha\cot^2\alpha-3\cot\alpha)+R_1a^2\cot\beta(\beta \cot \beta-1)\right]\varepsilon_0^2.
		$$
		Regarding the variation of the area of the symmetric difference, we have $|\Omega^\varepsilon\Delta B|=|\Omega\Delta B|+2\Delta A_e$, where $\Delta A_e$ is the variation of the external area.
			 
		Therefore 
		$$
		\mathcal{F}(\Omega_\varepsilon)=\dfrac{\dfrac{P+\Delta P-2\pi}{2\pi}}{\left(\dfrac{|\Omega\Delta B|+2\Delta A_e}{\pi}\right)^2}=
		\mathcal{F}(\Omega)\dfrac{1+N\dfrac{\Delta P}{2\pi\delta}}{\left(1+\dfrac{2N\Delta A_e}{\pi\lambda}\right)^2},
		$$
		and $\Omega$ is a minimizer of $\mathcal{F}$ if 
		$$
		\frac{\Delta P}{2\pi\delta}-4\frac{\Delta A}{\pi\lambda}-4N\frac{\Delta A_e}{\pi\lambda^2}\geq 0.
		$$
		It is easy to see that the $\varepsilon_0$ term of the previous quantity  is null, by using the optimality condition $(ii)$. 
		Moreover the $\varepsilon_0^2$ term is
		\begin{eqnarray*}
		&&\frac{1}{2\pi\delta}[R_0(\alpha +2\alpha \cot^2\alpha-2\cot\alpha)+\frac{R_0^2}{R_1}(\alpha+3\alpha \cot^2\alpha-3\cot\alpha)-R_1a^2\cot\beta(1-\beta \cot\beta)]+\\
		&&-\frac{4}{\pi\lambda}R_0^2(\alpha+3\alpha\cot^2\alpha-3\cot \alpha)-\frac{16NR_0^4}{\pi^2\lambda^2}(1-\alpha\cot\alpha)^2.
		\end{eqnarray*}
		Using again condition $(ii)$, we can say that the above quantity is positive if and only if $Q\geq 0$.
\end{proof}

\begin{lemma}\label{44}
		Let $\Omega$ be a connected planar set which minimizes the functional $\mathcal{F}$.
		Assume that $\Omega$ has a unique optimal ball and let $N$ be the order of its rotational symmetry. 
		If  $N\geq 4$, $\frac{\pi}{2N}\leq \theta\leq \frac{\pi}{N}$ and $\frac{\pi}{N}\leq \alpha\leq \pi$, then $\mathcal{A}_0-\mathcal{A}_1>0$.
\end{lemma}
\begin{proof}
	Let us set  $z(\theta,\alpha):=\mathcal{A}_0-\mathcal{A}_1$. It is easy to prove that  $\frac{\partial z}{\partial \alpha}>0$; therefore $z(\theta,\alpha)\geq z(\theta,\pi/N)$. 
	It is sufficient to prove that  $Z(\theta):=z(\theta,\pi/N)\geq 0$.
	We observe that $Z(\frac{\pi}{N})=0$ and $Z(\frac{\pi}{2N})>0$. 
	Notice that $Z'$ has only one zero  $\theta_0$ in $(0,\frac{\pi}{N})$, is positive in $(0,\theta_0)$ and negative in $(\theta_0,\frac{\pi}{N})$, therefore $Z$ is positive.
\end{proof}


In the sequel we will often use the following formula:
$$
\frac{1}{R_0}+\frac{1}{R_1}=\frac{\sin\left(\frac{\pi}{N}\right)\sin(\alpha-\theta)}{\sin(\theta)\sin\left(\frac{\pi}{N}-\theta\right)}.
$$
Notice that
$$
\mathcal{A}_0=\sin^2\theta h(\alpha) - g(\theta), \qquad \mathcal{A}_1=\sin^2\left(\theta-\frac{\pi}{N}\right) h\left(\alpha-\frac{\pi}{N}\right) + g\left(\frac{\pi}{N}-\theta\right)\,,
$$
where $g, h$ have been defined in (\ref{defn_fonctions_gh}).
\begin{lemma}\label{45}
	Let $\Omega$ be a connected planar set which minimizes the functional $\mathcal{F}$.
	Assume that $\Omega$ has a unique optimal ball and let $N$ be the order of its rotational symmetry. 
	If $N\geq 4$, $0\leq \theta\leq \frac{\pi}{2N}$ and $\frac{\pi}{N}\leq \alpha\leq \frac{\pi}{2}$, then $\mathcal{F}(\Omega)>0.406$.
\end{lemma}
\begin{proof}
	We observe that 
	$$
	\frac{1}{R_0}+\frac{1}{R_1}=\frac{\sin\left(\frac{\pi}{N}\right)\sin(\alpha-\theta)}{\sin(\theta)\sin\left(\frac{\pi}{N}-\theta\right)}\geq 2\cos\left(\frac{\pi}{2N}\right).
	$$ 
	 Proposition \ref{optimalconditions_connected} implies that
	\begin{equation}\label{prima}
	\frac{\delta(\Omega)}{\lambda(\Omega)}\geq \frac 14 \cos\left(\frac{\pi}{2N}\right),
	\end{equation}
	and 	$\lambda(\Omega)=\frac{N}{\pi}(\mathcal{A}_0+\mathcal{A}_1)=\frac{2N}{\pi}\mathcal{A}_0$.
	Notice that  
	\begin{eqnarray}\label{seconda}
	\lambda(\Omega)=\frac{2N}{\pi}\mathcal{A}_0
	&\leq& \frac{2N}{\pi}\left[\sin^2(\theta)\ h\left(\frac{\pi}{2}\right)-\theta+\sin(\theta)\cos(\theta)\right]
	\\\nonumber
	&\leq& \frac{2N}{\pi}\left[\sin^2 \left(\frac{\pi}{2N}\right) \frac{\pi}{2}-\frac{\pi}{2N}+\sin\left(\frac{\pi}{2N}\right)\cos\left(\frac{\pi}{2N}\right)\right].
	\end{eqnarray}
	By (\ref{prima}) and (\ref{seconda}) one has
	$\mathcal{F}(\Omega)>0.406$, since $N\geq 4$.
\end{proof}

\begin{rem}
	In the sequel we will often use the above technique to get a contradiction.
\end{rem}
\begin{lemma}\label{zone3}\label{47}
	Let $\Omega$ be a connected planar set which minimizes the functional $\mathcal{F}$.
	Assume that $\Omega$ has a unique optimal ball and let $N$ be the order of its rotational symmetry. 
	If $N\geq 4$, $0\leq \theta\leq \frac{\pi}{2N}$ and $\frac{\pi}{2}\leq \alpha\leq  \frac{\pi}{2}+\frac{\pi}{N}$, then $Q<0$,where $Q$ is  defined in Proposition \ref{optimalconditions_connected}.
\end{lemma}
\begin{proof}
Since $H(\alpha)\leq 0$ and	$H(\alpha-\frac{\pi}{N})\geq 0$, the quantity $Q$ is negative.
\end{proof}

\begin{lemma}\label{Ngrand}\label{48}
	Let $\Omega$ be a connected planar set which minimizes the functional $\mathcal{F}$.
	Assume that $\Omega$ has a unique optimal ball and let $N$ be the order of its rotational symmetry. 
	If $N\geq 17$, $0\leq \theta\leq \frac{\pi}{2N}$ and $\frac{3\pi}{4}\leq \alpha\leq \pi$, then $\mathcal{F}(\Omega)>0.406$.
\end{lemma}
\begin{proof}
One has
$\delta=\frac{N}{\pi}d(\alpha,\theta)-1$,	
where 
\begin{equation}\label{definition_d}
	d(\alpha,\theta):=\sin(\theta)\frac{\alpha}{\sin(\alpha)}+\sin\left(\frac{\pi}{N}-\theta\right)\frac{\alpha-\frac{\pi}{N}}{\sin(\alpha-\frac{\pi}{N})}\,.
	\end{equation}
		Observe that $d(\alpha,\theta)\geq d(\alpha,0)$, since
	$\frac{\partial d}{\partial \theta}\geq 0$, and 
	$$
	d(\alpha,0)\geq \sin\left(\frac{\pi}{N}\right)\frac{\frac 34 \pi - \frac{\pi}{N}}{\sin(\frac 34 \pi - \frac{\pi}{N})}.
	$$
	By Proposition \ref{optimalconditions_connected}, and the fact that $\lambda(\Omega)\leq 2$, we have
	\begin{equation}\label{minimum1}
	\mathcal{F}(\Omega)\geq m_1:=\frac 14\left[\frac{N}{\pi}\sin\left(\frac{\pi}{N}\right)\frac{\frac 34 \pi - \frac{\pi}{N}}{\sin(\frac 34 \pi - \frac{\pi}{N})}\right],
	\end{equation}
	and $m_1>0.406$ for $N\ge 17$.
\end{proof}

\begin{lemma}\label{49}
	Let $\Omega$ be a connected planar set which minimizes the functional $\mathcal{F}$.
	Assume that $\Omega$ has a unique optimal ball and let $N$ be the order of its rotational symmetry. 
	If $N\geq 8$, $0\leq \theta\leq \frac{\pi}{2N}$ and $\frac{\pi}{2}+\frac{\pi}{N}\leq \alpha\leq  \frac{3\pi}{4}$, then $\mathcal{F}(\Omega)>0.406$.
\end{lemma}
\begin{proof}As in the previous lemma,
	$\delta(\Omega)=\frac{N}{\pi}d(\alpha,\theta)-1$, where $d(\alpha,\theta)$ has been defined in (\ref{definition_d}) and
	$d(\alpha,\theta)\geq d(\alpha,0)\geq \frac{\pi}{2}\sin(\frac{\pi}{N})$.
	Moreover $\lambda(\Omega)\leq \frac{2N}{\pi}[\sin^2(\frac{\pi}{2N})h(\frac{3\pi}{4})-g(\frac{\pi}{2N})]$.
	Hence
	\begin{equation}\label{minimum2}
	\mathcal{F}(\Omega)=\frac{\delta(\Omega)}{\lambda^2(\Omega)}\geq m_2:=\frac{\frac{N}{2}\sin(\frac{\pi}{N})-1}{\frac{4N^2}{\pi^2}[\sin^2(\frac{\pi}{2N})h(\frac{3\pi}{4})-g(\frac{\pi}{2N})]^2},
	\end{equation}
	and $m_2>0.406$ for $N\ge 8$.
\end{proof}
\begin{rem}
	Notice that estimates (\ref{minimum1}) and (\ref{minimum2}) for  $\mathcal{F}$ hold for every $N\geq 4$. 
	However they are not sufficient to conclude for any $N\ge 3$.
\end{rem}

\begin{lemma}\label{411}
	Let $\Omega$ be a connected planar set which minimizes the functional $\mathcal{F}$.
	Assume that $\Omega$ has a unique optimal ball and let $N$ be the order of its rotational symmetry. 
	If either 
	
	(i)	$4\leq N\leq 16$, $0\leq \theta\leq \frac{\pi}{2N}$ and $\frac{3\pi}{4}\leq \alpha \leq \pi$, 

\noindent or
	
	(ii) 	$5\leq N\leq 7$, 	$0\leq \theta\leq \frac{\pi}{2N}$ and $\frac{\pi}{2}+\frac{\pi}{N}\leq \alpha\leq  \frac{3\pi}{4}$,
	
\noindent then $Q<0$, where  $Q$ is defined in Proposition \ref{optimalconditions_connected}.
	
\end{lemma}
\begin{proof}
	We are assuming, in both cases, that $\alpha\geq \frac{\pi}{2}+\frac{\pi}{N}$. If $H(\alpha)\leq H(\alpha-\frac{\pi}{N})<0$, then $Q$ is negative, because the sum of its first two terms is negative. Therefore we can assume that $H(\alpha-\frac{\pi}{N})<H(\alpha)<0$. 
	The mean value theorem and the estimate $|H'|\leq 0.3$ on the interval $(1.973,\pi)$ imply that $-H(\alpha-\frac{\pi}{N})\leq 0.3\frac{\pi}{N}-H(\alpha)$. Therefore, dropping negative terms in the expression of $Q$, we get
	\begin{eqnarray*}
	Q&\leq& H(\alpha)\left[\frac{1}{\sin^3\theta}-\frac{1}{\sin^3(\frac{\pi}{N}-\theta)}\right]
	+\frac{0.3\pi}{N\sin^3(\frac{\pi}{2N})}-\frac{32 N}{\pi}\frac{\delta(\Omega)}{\lambda^2(\Omega)}\\
	&\leq& \frac{0.3\pi}{N\sin^3(\frac{\pi}{2N})}-\frac{32 N}{\pi}\frac{\delta(\Omega)}{\lambda^2(\Omega)}\,.
	\end{eqnarray*}
	We are now going to use estimates (\ref{minimum1}) and (\ref{minimum2}). 
	If assumption $(i)$ holds we have
		$
		Q\leq \frac{0.3\pi}{N\sin^3(\frac{\pi}{2N})}-\frac{32 N}{\pi}m_1<0
		$ for $4\leq N\leq 16$.
	If assumption $(ii)$ holds true, we have 
		$
		Q\leq \frac{0.3\pi}{N\sin^3(\frac{\pi}{2N})}-\frac{32 N}{\pi}m_2<0
		$ for $5\leq N\leq 7$.
\end{proof}

We are now left with the cases $N=2, 3$.

\begin{lemma}\label{412}
		Let $\Omega$ be a connected planar set which minimizes the functional $\mathcal{F}$.
		Assume that $\Omega$ has a unique optimal ball and let $N=3$ be the order of its rotational symmetry.
	\begin{enumerate}
		\item
		If $0\leq \theta\leq \frac{\pi}{6}$ and $\frac{\pi}{3}\leq \alpha\leq \frac{\pi}{2}$, then $\mathcal{F}(\Omega)>0.406$.
		\item
		If $0\leq \theta\leq \frac{\pi}{6}$ and $\frac{\pi}{2}\leq \alpha\leq \frac{5\pi}{6}$, then $Q<0$.
		\item
		If $0\leq \theta\leq \frac{\pi}{12}$ and $\frac{5\pi}{6}\leq \alpha\leq \pi$, then $\mathcal{A}_1>\frac{\pi}{3}$.
		\item
		If $\frac{\pi}{12}\leq \theta\leq \frac{\pi}{6}$ and $\frac{5\pi}{6}\leq \alpha\leq \pi$, then $Q<0$.
		\item
		If $\frac{\pi}{6}\leq \theta\leq \frac{\pi}{3}$ and $\frac{\pi}{3}\leq \alpha\leq \pi$, then $\mathcal{A}_0-\mathcal{A}_1>0$.
		
	\end{enumerate}
\end{lemma}
\begin{proof}

	(1). Assume $\frac{5\pi}{12}\leq \alpha\leq \frac{\pi}{2}$. Then
		$$
		\frac{8\delta}{\lambda}\,=\frac{1}{R_0}+\frac{1}{R_1}=\frac{\sin(\frac{\pi}{3})\sin(\alpha-\theta)}{\sin(\theta)\sin(\frac{\pi}{3}-\theta)}\geq \sqrt{6}.
		$$
		Moreover, 
		$$
		\lambda(\Omega)=\frac{6}{\pi}\mathcal{A}_0\leq \frac{6}{\pi}\left[\sin^2\left(\frac{\pi}{6}\right)h\left(\frac{\pi}{2}\right)-g\left(\frac{\pi}{6}\right)\right].
		$$
		Therefore, by Proposition \ref{optimalconditions_connected}, $\mathcal{F}(\Omega)>0.406$.
		
		Now assume $\frac{\pi}{3}\leq \alpha\leq \frac{5\pi}{12}$. Then
		$$
		\frac{8\delta}{\lambda}\,=\frac{1}{R_0}+\frac{1}{R_1}=\frac{\sin(\frac{\pi}{3})\sin(\alpha-\theta)}{\sin(\theta)\sin(\frac{\pi}{3}-\theta)}\geq \sqrt{3}\,.
		$$
Moreover, 
		$$
		\lambda(\Omega)=\frac{6}{\pi}\mathcal{A}_0\leq \frac{6}{\pi}\left[\sin^2\left(\frac{\pi}{6}\right)h\left(\frac{5\pi}{12}\right)-g\left(\frac{\pi}{6}\right)\right]\,.
		$$
		Therefore, by Proposition \ref{optimalconditions_connected}, $\mathcal{F}(\Omega)>0.406$.
		
	(2).
		With the same arguments of  the proof of Lemma \ref{zone3}, one obtains  $Q<0$.
	
	(3).
		$\mathcal{A}_1=\sin^2(\theta-\frac{\pi}{3}) h(\alpha-\frac{\pi}{3}) + g(\frac{\pi}{3}-\theta)\geq \sin^2(\frac{\pi}{3}-\frac{\pi}{12})h(\frac{5\pi}{6}-\frac{\pi}{3})+g(\frac{\pi}{3}-\frac{\pi}{12})>\frac{\pi}{3}$.
	
	(4).
		Let $d(\alpha,\theta)$ be defined by (\ref{definition_d}). As in Lemma \ref{Ngrand}, we have
		$d(\alpha,\theta)\geq d(\alpha,0)\geq 
		d(\frac{5\pi}{6},\frac{\pi}{12})$. Using that  $\lambda(\Omega)\leq 2$, one gets
		$Q\leq \frac{0.3 \pi}{3\sin^3(\pi/6)}-\frac{32\cdot 3}{\pi}\frac 14 (\frac{3}{\pi}d(\frac{5\pi}{6},\frac{\pi}{12})-1)<0$.
		
	(5). The proof is completely analogous to that one of the case $N\ge 4$.	
\end{proof}

\begin{lemma}\label{413}
		Let $\Omega$ be a connected planar set which minimizes the functional $\mathcal{F}$.
		Assume that $\Omega$ has a unique optimal ball and let $N=2$ be the order of its rotational symmetry. 
	Then $Q<0$, where $Q$ is defined in Proposition \ref{optimalconditions_connected}. 
\end{lemma}
\begin{proof}
	Notice that $H(\alpha)\leq 0$ for $\frac{\pi}{2}\leq \alpha<\pi$ and $H(\alpha-\pi/2)> 0$ for $\alpha<{\pi}$.
	Therefore $Q<0$.
\end{proof}


\subsubsection{Non-connected case}\label{non-connected-section}
Let us first remark that we can assume that the optimal domain has, in this case, only two connected components.
Indeed, let $E_1$ be the connected component which intersects the optimal ball and assume we have several other connected
components $E_2, E_3, \ldots$. First, we can assume that all these components have an empty intersection with the optimal 
ball $B$, since otherwise one can translate 
the components far away, increasing the value of $\lambda$ and keeping the value of $\delta$, by Lemma \ref{comp_conn_loin}. 
Moreover, by the isoperimetric inequality, we can obviously assume that all of these components $E_k,k\geq 2$ are balls. 
Then, we can apply the optimization
procedure explained in the proof of Proposition \ref{perimeter} for which we have seen that the optimal configuration
is composed of one or two balls. But if there are two balls, they must have different radii which is not possible
since the optimality condition (2) of Proposition \ref{thmqualit} claims that the radii outside the optimal balls
must be equal. Thus, the optimal domain must have only two connected components $E_1$, $E_2$ and $E_2$ is a ball not
intersecting $B$.

If the set $E_1$ is contained into $B$, then $E_1$ is a ball. Hence, by a direct computation
for the union of two balls, it holds $\mathcal{F}(\Omega)\ge 0.406$, and so $\Omega$ cannot be optimal.

We therefore assume that $E_1$ is not contained into $B$. In this case, the parametrization of its boundary is the same as in the connected case (see Figure \ref{thetaeta}), and we have 
$0\leq \theta\leq \frac{\pi}{N}$ and 
$\theta< \alpha\leq \pi$. 
Indeed, the condition 
$$
1< x_A+R_0=\frac{\sin(\alpha-\theta)}{\sin \alpha}+\frac{\sin \theta}{\sin \alpha},
$$
where $A$ is defined as in Figure \ref{thetaeta}, is equivalent to $\alpha > \theta$. 
Moreover we observe that 
\begin{equation}\label{R0R1-NONconn}
R_1=
\frac{\sin(\frac{\pi}{N}-\theta)}{\sin(\frac{\pi}{N}- \alpha)}, \qquad
R_0=
\frac{\sin(\theta)}{\sin(\alpha)}.
\end{equation} 
We remark that $E_2$ is a ball with radius $R_0$. 
Notice that
\begin{equation}\label{A0A1-NONconn}
\mathcal{A}_0=R_0^2 g(\alpha) - g(\theta), \qquad\mathcal{A}_1=R_1^2 g\left(\alpha-\frac{N}{\pi}\right) + g\left(\frac{N}{\pi}-\theta\right),
\end{equation}
and
$$
\lambda(\Omega)=\frac{2N\mathcal{A}_0}{\pi}+2R_0^2,\qquad
\delta(\Omega)=
\frac{N}{\pi}\left[\alpha \frac{\sin \theta}{\sin \alpha}+ \left(\alpha-\frac{N}{\pi}\right)\frac{\sin(\frac{N}{\pi}-\theta)}{\sin(\alpha - \frac{N}{\pi})}\right]+R_0-1.
$$ 
As in the above subsection, the proof will be splitted into several parts, according to the values $\theta, \alpha, N$. For every part, the contradiction will be given by the fact that $\Omega$ does not satisfy one of the conditions 
expressed in the next proposition.
\begin{prop}\label{optimalconditions_NONconnected}
		Let $\Omega$ be a non-connected planar set which minimizes the functional $\mathcal{F}$.
		Assume that $\Omega$ has a unique optimal ball $B$. 
		Let $w$ be the unique connected component which intersects $B$ and let $N$ be the order of its rotational symmetry. 
	Then the following conditions hold.
	\begin{enumerate}
		\item[(i)]
		$\mathcal{A}_0-\mathcal{A}_1+\frac{\pi}{N}R_0^2=0$, where $\mathcal{A}_0$ and $\mathcal{A}_1$ are defined in (\ref{A0A1-NONconn});
		\item[(ii)]
		$
		\frac{1}{R_0}+\frac{1}{R_1}=\frac{8\delta(\Omega)}{\lambda(\Omega)}
		$, where $R_0$ and  $R_1$ are defined in (\ref{R0R1-NONconn});
		\item[(iii)]
		$
		\mathcal{F}(\Omega)< 0.406;
		$
		\item[(iv)] 
		$
		Q\geq 0
		$, where $Q:=\frac{1}{\sin^3(\theta)}H(\alpha)-\frac{1}{\sin^3(\pi/N-\theta)}H(\alpha-\pi/N)-\frac{32 N}{\pi}\frac{\delta}{\lambda^2}$, $H(x)=\frac{\sin^3(x)}{\tan(x)-x}$;
		\item[(v)]
		$
		\Phi(\alpha)\geq 0
		$, where $\Phi(\alpha):=NR_0[1-\alpha \cot(\alpha)][\cot(\alpha)-\frac{N}{\pi}(1-\alpha \cot(\alpha))]$.
	\end{enumerate}
\end{prop}
\begin{proof}
	Statements in $(i)$-$(iv)$ are analogous to those of the connected case (see Proposition  \ref{optimalconditions_connected}). 
	Let us prove statement $(v)$.
	We consider the following perturbation of $\Omega$: we reduce the area of the second component $E_2$, increasing that one of $E_1\setminus B$ in order to keep the total area  equal to $\pi$. Therefore 
	$\alpha$ is modified into $\alpha+\varepsilon$ and $R_0$ into $R_0^\varepsilon$. Observe that $R_1$ keeps unchanged.
	Now, 
	$$
	R_0^\varepsilon=\frac{\sin \theta}{\sin \alpha}+\varepsilon \cos \alpha-\frac{\varepsilon^2}{\sin \alpha}=R_0[1-\varepsilon \cot \alpha +\varepsilon^2(1/2+\cot^2 \alpha)],
	$$
	and
	$$
	(R_0^\varepsilon)^2=R_0^2[1-2 \varepsilon \cot \alpha +\varepsilon^2(1+3\cot^2 \alpha)].
	$$
	Moreover $\mathcal{A}_0^\varepsilon=(R_0^\varepsilon)^2g(\alpha+\varepsilon)-g(\theta)$ and $g(\alpha+\varepsilon)=g(\alpha)+\varepsilon (1-\cos(2\alpha))+\varepsilon^2\sin(2\alpha)$.
	Therefore 
	$$
	\mathcal{A}_0^\varepsilon=\mathcal{A}_0+\varepsilon R_0^2 2(1-\alpha \cot \alpha)+\varepsilon^2 R_0^2 [\alpha+3\alpha \cot^2\alpha -3\cot \alpha].
	$$
	Since we are keeping the total area unchanged, we have $\pi R_0^2-N (\mathcal{A}_0^\varepsilon-\mathcal{A}_0)=\pi (R_0^\varepsilon)^2$
	which gives 
	$$
	R_0^\varepsilon=R_0-\frac 12 \frac{N (\mathcal{A}_0^\varepsilon-\mathcal{A}_0)}{\pi R_0}-\frac 18 \frac{N^2 (\mathcal{A}_0^\varepsilon-\mathcal{A}_0)^2}{\pi R_0^3}, 
	$$
	and hence
	$$
	2\pi(R_0^\varepsilon-R_0)=-N\left[2\varepsilon(1-\alpha \cot \alpha)+\varepsilon^2((\alpha+3\alpha \cot^2\alpha -3\cot \alpha)+\frac{N}{\pi}(1-\alpha \cot \alpha)^2)\right].
	$$
	The variation of the perimeter equals to
	\begin{eqnarray*}
	&&2N(\alpha +\varepsilon)R_0\left[1-\varepsilon \cot \alpha +\varepsilon^2(1/2+\cot^2 \alpha)\right] - 2N\alpha R_0 +\\
	&&-N\left[2\varepsilon(1-\alpha \cot \alpha)+\varepsilon^2(\alpha+3\alpha \cot^2\alpha -3\cot \alpha+\frac{N}{\pi}(1-\alpha \cot \alpha)^2)\right].
	\end{eqnarray*}
	Observe that the previous expression is a second order polynomial in $\varepsilon$. 
	In particular the terms in $\varepsilon$ vanish according to (ii). The terms in $\varepsilon^2$ equal to 
	$NR_0\Phi(\alpha)=NR_0[1-\alpha \cot(\alpha)][\cot(\alpha)-\frac{N}{\pi}(1-\alpha \cot(\alpha))]$
	and must be positive by the optimality of the set $\Omega$.
\end{proof}

\begin{lemma}\label{propalpha}\label{415}
		Let $\Omega$ be a non-connected planar set which minimizes the functional $\mathcal{F}$.
		Assume that $\Omega$ has a unique optimal ball $B$. 
		Let $w$ be the unique connected component which intersects $B$ and let $N\ge 2$ be the order of its rotational symmetry. 
	 Hence the function $\Phi(\alpha)$  defined in Proposition \ref{optimalconditions_NONconnected} has a zero $\alpha(N)>\frac{\pi}{N}$ such that  $\Phi(\alpha)<0$ for $\alpha(N)<  \alpha \leq \pi$.
	In particular $\alpha(2)\approx 1.22$ and
	$\alpha(N)<1$ for $N\geq 6$.
\end{lemma}

\begin{lemma}\label{416}
		Let $\Omega$ be a non-connected planar set which minimizes the functional $\mathcal{F}$.
		Assume that $\Omega$ has a unique optimal ball $B$. 
		Let $w$ be the unique connected component which intersects $B$ and let $N$ be the order of its rotational symmetry. 
		
	If $N\ge 3$, $\frac{\pi}{2N}\leq \theta \leq \frac{\pi}{N}$ and $\theta\leq \alpha\leq \pi$, then $\mathcal{A}_0-\mathcal{A}_1 +\frac{\pi}{N}R_0^2>0$.
\end{lemma}
\begin{proof}
	We divide the proof into two parts, according to the values of $\alpha$.
		If $\frac{\pi}{N}\leq \alpha\leq \pi$, one proves, as in the connected case, that  $\mathcal{A}_0-\mathcal{A}_1>0$.
		If $\theta< \alpha\leq \frac{\pi}{N}$ we have
		$$\mathcal{A}_0-\mathcal{A}_1={\sin^2 \theta}h(\alpha)-g(\theta) + {\sin^2 \left(\frac{\pi}{N}-\theta\right)}h\left(\frac{\pi}{N}-\alpha\right)-g\left(\frac{\pi}{N}-\theta\right)\,.$$ Let us set $z(\theta,\alpha)$   the right hand of the last equality. The derivative with respect to $\alpha$ of $z$ is positive. Therefore
		$\mathcal{A}_0-\mathcal{A}_1> z(\theta,\theta)=0$.
\end{proof}

\begin{lemma}\label{417}
		Let $\Omega$ be a non-connected planar set which minimizes the functional $\mathcal{F}$.
		Assume that $\Omega$ has a unique optimal ball $B$. 
		Let $w$ be the unique connected component which intersects $B$ and let $N$ be the order of its rotational symmetry. 
		
If $N\geq 3$, $0\leq \theta\leq \frac{\pi}{2N}$ and $\frac{\pi}{N}\leq \alpha\leq \alpha(N)$, then $\mathcal{F}(\Omega)>0.406$.
\end{lemma}
\begin{proof}
The proof is composed by several steps.

\emph{Step 1.}
		Assume that  $\frac{\pi}{N}\leq \alpha\leq 1$ and $N\geq 10$. By Proposition \ref{optimalconditions_NONconnected}
		$$
		\frac{8\delta(\Omega)}{\lambda(\Omega)}=\frac{1}{R_0}+\frac{1}{R_1}\geq 2\cos\left(\frac{\pi}{2N}\right).
		$$ 
		Moreover
		$$
		\lambda(\Omega)=\frac{2N}{\pi}\mathcal{A}_0+2R_0^2\leq 
		\left[\sin^2 \left(\frac{\pi}{2N}\right) h(1)-\frac{\pi}{2N}+\frac 12 \sin\left(\frac{\pi}{2N}\right)\right]
		\frac{2N}{\pi}+\frac{1}{2\cos^2(\frac{\pi}{2N})},
		$$
		and hence $\mathcal{F}(\Omega)>0.406$.
		
\emph{Step 2.}		
		Assume that  $\frac{\pi}{N}\leq \alpha\leq \frac{\pi}{N-1}$ and $N=7, 8, 9$. By Proposition \ref{optimalconditions_NONconnected},
		$$
		\frac{8\delta(\Omega)}{\lambda(\Omega)}=\frac{1}{R_0}+\frac{1}{R_1}\geq 2\cos\left(\frac{\pi}{2N}\right).
		$$
		Moreover, 
		\begin{eqnarray*}
		\lambda(\Omega) &=&\frac{2N}{\pi}\mathcal{A}_0+2R_0^2\\
		&\le&
		\left[\sin^2 \left(\frac{\pi}{2N}\right) h\left(\frac{\pi}{N-1}\right)-\frac{\pi}{2N}+\frac 12 \sin\left(\frac{\pi}{2N}\right)\right]\frac{2N}{\pi}+\frac{1}{2\cos^2(\frac{\pi}{2N})},
		\end{eqnarray*}
		and hence $\mathcal{F}(\Omega)>0.406$.

\emph{Step 3.}
		Assume that  $\frac{\pi}{N-1}\leq \alpha\leq 1$ and $N=7, 8, 9$. By Proposition \ref{optimalconditions_NONconnected},
		$$
		\frac{8\delta(\Omega)}{\lambda(\Omega)}=\frac{1}{R_0}+\frac{1}{R_1}\geq 2\cot\left(\frac{\pi}{2N}\right)\sin\left(\frac{\pi}{N-1}-\frac{\pi}{2N}\right).
		$$
		Moreover, 
		\begin{eqnarray*}
		\lambda(\Omega) &=&\frac{2N}{\pi}\mathcal{A}_0+2R_0^2\\
		&\le& 
		\left[\sin^2 \left(\frac{\pi}{2N}\right) h\left(\frac{\pi}{N-1}\right)-\frac{\pi}{2N}+\frac 12 \sin\left(\frac{\pi}{2N}\right)\right]
		\frac{2N}{\pi}+\frac{1}{2\cos^2(\frac{\pi}{2N})},
	\end{eqnarray*}
		therefore
		$\mathcal{F}(\Omega)>0.406$.
		
\emph{Step 4.}
		Assume that  $\frac{\pi}{N}\leq \alpha\leq \alpha(N)$, $0\leq \theta\leq \frac{\pi}{2N+2}$ and $N=3, 4, 5, 6$. By Proposition \ref{optimalconditions_NONconnected},
		$$
		\frac{8\delta(\Omega)}{\lambda(\Omega)}=\frac{1}{R_0}+\frac{1}{R_1}\geq \frac{\sin(\frac{\pi}{N})}{\sin(\frac{\pi}{2N+2})}.
		$$ 
		Moreover, 
		\begin{eqnarray*}
		\lambda(\Omega) &=&\frac{2N}{\pi}\mathcal{A}_0+2R_0^2\\
		&\le& 
		\left[\sin^2 \left(\frac{\pi}{2N+2}\right) h\left(\frac{\pi}{2N+2}\right)-g\left(\frac{\pi}{2N+2}\right)\right]
		\frac{2N}{\pi}+2\frac{\sin^2\left(\frac{\pi}{2N+2}\right)}{\sin^2\left(\frac{\pi}{N}\right)},
	\end{eqnarray*}
		therefore
		$\mathcal{F}(\Omega)>0.406$.
		
\emph{Step 5.}
		Assume that  $\frac{\pi}{N}\leq \alpha\leq \alpha(N)$, $\frac{\pi}{2N+2}\leq \theta\leq \frac{\pi}{2N}$ and $N=3, 4, 5, 6$. Then
		$$
		\delta(\Omega)\geq \frac{N}{\pi}\left[\sin\left(\frac{\pi}{2N+2}\right)\frac{\frac{\pi}{N}}{\sin\left(\frac{\pi}{N}\right)}+\sin\left(\frac{\pi}{N}-\frac{\pi}{2N+2}\right)\right]+\frac{\sin\left(\frac{\pi}{2N+2}\right)}{\sin(\alpha(N))}-1.
		$$ 
		Moreover, 
		$$
		\lambda(\Omega)=\frac{2N}{\pi}\mathcal{A}_0+2R_0^2
		\leq 
		\left[\sin^2 \left(\frac{\pi}{2N}\right) h\left(\alpha(N)\right)-g\left(\frac{\pi}{2N}\right)\right]\frac{2N}{\pi}+\frac{2\sin^2(\frac{\pi}{2N})}{\sin^2(\frac{\pi}{N})},
		$$
		therefore
		$\mathcal{F}(\Omega)>0.406$.
\end{proof}

\begin{lemma}\label{418}
		Let $\Omega$ be a non-connected planar set which minimizes the functional $\mathcal{F}$.
		Assume that $\Omega$ has a unique optimal ball $B$. 
		Let $w$ be the unique connected component which intersects $B$ and let $N$ be the order of its rotational symmetry. 
	
	If $N\geq 3$, $0\leq \theta\leq \frac{\pi}{2N}$ and {$\theta\leq \alpha\leq \frac{\pi}{N}$}, then $\mathcal{F}(\Omega)>0.406$.
\end{lemma}
\begin{proof}
	Let us estimate $\delta(\Omega)$ and $\la(\Omega)$; for every $N$ we have
	\begin{eqnarray*}
	\delta(\Omega) &\geq& \frac{N}{\pi}\left[\theta+\sin\left(\frac{\pi}{N}-\theta\right)\right]+R_0-1\geq\frac{N}{\pi}\sin\left(\frac{\pi}{N}\right)+R_0-1,\\
	\lambda(\Omega) &=& \frac{2N}{\pi}\mathcal{A}_0+2R_0^2\leq \frac{2N}{\pi}\left[\frac{\sin^2 \theta}{\sin^2 \alpha}g(\alpha)-g(\theta)\right]+2R_0^2\leq \frac{2N}{\pi}R_0^2 g\left(\frac{\pi}{N}\right)+2R_0^2\,.
\end{eqnarray*}
	Let $a_N=1+\frac{N}{\pi}g(\frac{\pi}{N})$,  $b_N=\frac{N}{\pi}\sin(\frac{\pi}{N})-1$. One has
	\begin{equation}\label{a}
	\lambda(\Omega)\leq 2R_0^2a_N, 
	\end{equation}
	and hence
	$$
	\mathcal{F}(\Omega)=\frac{\delta(\Omega)}{\lambda^2(\Omega)}\geq \frac{b_N+R_0}{4R_0^4 a_N^2}.
	$$ 
	Notice that the function 
	$$
	x\mapsto \frac{b_N+x}{4x^4 a_N^2},
	$$ 
	attains its maximum at $x=-\frac 43 b_N$. Therefore
	$\mathcal{F}(\Omega)>0.406$ if $R_0\geq -\frac{4}{3}b_N$ and $N\geq 5$.  	
	
	\noindent
	Now, let $N=3,4$ and $R_0\geq -\frac{4}{3}b_N$. We divide the analysis in two parts, according to the values of $\alpha$.

	Assume that $\alpha<\frac{\pi}{N}$.  We can write $\delta(\Omega)=\frac{N}{\pi}d(\alpha, \theta)+R_0-1$, where $d$ has been defined in (\ref{definition_d}). We observe that $\frac{\partial d}{\partial \alpha}<0$. Therefore
	$$
	d(\alpha,\theta)\geq d\left(\frac{\pi}{2N},\theta\right) \geq \frac{\pi}{N}\cos\left(\frac{\pi}{2N}\right).
	$$
By estimate (\ref{a}) of  $\lambda(\Omega)$ we get
	$$
	\mathcal{F}(\Omega)\geq \frac{\cos(\frac{\pi}{2N})-1+R_0}{4R_0^4a_N^2}.
	$$
	Notice that the right hand side is a decreasing function of $R_0$ on $[-\frac 43 b_N, \frac{\sqrt{2}}{2}]$ and its value at $R_0=\frac{\sqrt{2}}{2}$ is greater than $0.406$.
	
	Assume now that $\alpha\geq \frac{\pi}{N}$. We have $\frac{\partial d}{\partial \theta}\geq 0$. Since we are studying $R_0\geq -\frac{4}{3}b_N$, we are reduced to compute the minimum of $d(\arcsin(\frac{\sin \theta}{-\frac 43 b_N}),\theta)$. The minimum of this function is equal to $0.9918$ if $N=3$ and $0.7630$ if $N=4$. Therefore
	$$
	\delta(\Omega)
	\geq 
	-\frac 43 b_N-1+\frac{N}{\pi}\cdot
	\left\{
	\begin{array}{ll}
	0.9918, & N=3,
	\\
	0.7630, & N=4.
	\end{array}
	\right.
	$$
	We are now going to estimate $\lambda(\Omega)$. Since $R_0\geq -\frac 43 b_N$ and $\alpha\geq \frac{\pi}{2N}$, then $\theta\geq \widehat{\theta}_N$, where $\widehat{\theta}_N$ is defined by 
	$$
	\frac{\sin \widehat{\theta}_N}{\sin \frac{\pi}{2N}}=-\frac 43 b_N.
	$$
	 Therefore 
	 $$
	 \lambda(\Omega)=\frac{2N}{\pi}\mathcal{A}_1\leq \frac{2N}{\pi}g\left(\frac{\pi}{N}-\hat{\theta}_N\right).
	 $$
	One can easily verify that $\mathcal{F}(\Omega)>0.406$.

	We are now going to study the case $R_0\leq -\frac 43 b_N$.
	Let $b<1$ and assume that $R_0<b$. Then $\sin(\alpha-\theta)\geq \sin \theta \cos \theta (\frac 1b-1)$. Therefore 
	\begin{equation}\label{b}
	8\frac{\delta(\Omega)}{\lambda(\Omega)}\geq \frac{\sin(\frac{\pi}{N})\sin(\alpha-\theta)}{\sin \theta \sin(\frac{\pi}{N}-\theta)}\geq \frac 1b-1\qquad\forall\,N.
	\end{equation}
	For $b=-\frac 43 b_N$, one has $\mathcal{F}(\Omega)>0.406$, by estimate (\ref{a}); that is, for $N\geq 3$, if $R_0\leq -\frac 43 b_N$, then $\frac{\delta(\Omega)}{\lambda^2(\Omega)}>0.406$.
\end{proof}

\begin{lemma}\label{419}
	 Let $\Omega$ be a non-connected planar set which minimizes the functional $\mathcal{F}$.
	 Assume that $\Omega$ has a unique optimal ball $B$. 
	 Let $w$ be the unique connected component which intersects $B$ and let $N=2$ be the order of its rotational symmetry. 
	\begin{enumerate}
		\item[(1)]
		If $\alpha>\alpha(2)$, then $\Phi(\alpha)<0$, where $\Phi(\cdot)$ is defined in Proposition \ref{optimalconditions_NONconnected}.
		\item[(2)]
		If $\frac{\pi}{4}\leq \theta\leq \frac{\pi}{2}$ and $\theta\leq \alpha\leq \frac{\pi}{2}$, then $\mathcal{A}_0-\mathcal{A}_1+\frac{\pi}{2}R_0^2>0$.
		\item[(3)]
		If $\alpha\leq \frac{\pi}{6}$, then $\mathcal{F}(\Omega)>0.406$.
		\item[(4)]
		If $0\leq \theta\leq \frac{\pi}{4}$,
		$\frac{\pi}{6}\leq \alpha\leq \alpha(2)$ and $R_0\leq 0.45$, then $\mathcal{A}_0-\mathcal{A}_1+\frac{\pi}{2}R_0^2<0$.
		\item[(5)]
		If $0\leq \theta\leq \frac{\pi}{4}$, $\frac{\pi}{6}\leq \alpha\leq \alpha(2)$ and 
		$R_0\geq 0.45$, then $\mathcal{F}(\Omega)>0.406$.
	\end{enumerate}
\end{lemma}
\begin{proof}
	$(1).$ This  follows from Proposition \ref{propalpha}.
	
$(2).$ 	The proof is analogous to that of the case $N\ge 3$, see Lemma \ref{416}.
	
$(3).$
		Let $\theta\leq \alpha\leq \frac{\pi}{2}$ and
		let $R_0<0.3670$. One can use estimates (\ref{a}) and (\ref{b})  with $b=0.3670$ to prove that
		$\frac{\delta}{\lambda^2}>0.406$. 
		We now assume that $R_0\geq 0.3670$.
		Let us estimate $\delta(\Omega)$: using the function $d$ defined in (\ref{definition_d}) we get
		$$
		\delta(\Omega)=\frac{2}{\pi}d(\alpha,\theta)+R_0-1
		\geq \frac{2}{\pi}d\left(\frac{\pi}{6},\theta\right)+R_0-1\geq \frac{2}{\pi}\frac{\frac{\pi}{3}}{\sin\frac{\pi}{3}}+R_0-1,
		$$
		because the derivative of $d$ with respect to $\alpha$ is negative.
		On the other hand, 
		$$
		\lambda(\Omega)=\frac{4}{\pi}\mathcal{A}_0+2R_0^2\leq \frac{4}{\pi}R_0^2 g\left(\frac{\pi}{6}\right) +2 R_0^2,
		$$ 
		therefore
		$$
		\frac{\delta(\Omega)}{\lambda^2(\Omega)}\geq \frac{\frac{2}{\pi}\frac{\frac{\pi}{3}}{\sin\frac{\pi}{3}}+R_0-1}{\frac{4}{\pi}R_0^2 g(\frac{\pi}{6}) +2 R_0^2}.
		$$
		The function on the right hand side being decreasing with respect to $R_0$, if $R_0\in (0.3670,\frac{\sqrt{2}}{2})$, it holds
		$$
		\mathcal{F}(\Omega)=\frac{\delta}{\lambda^2}\geq \frac{\frac{2}{\pi}\frac{\frac{\pi}{3}}{\sin\frac{\pi}{3}}+\frac{\sqrt{2}}{2}-1}{\frac{2}{\pi} g(\frac{\pi}{6}) +1}>0.406.
		$$
		
	$(4).$
		We are going to prove that $\mathcal{A}_0-\mathcal{A}_1+\frac{\pi}{2}R_0^2<0$, that is, 
		$$
		G(\alpha,\theta):=\frac{\sin^2 \theta}{\sin^2 \alpha}\left[g(\alpha)+\frac{\pi}{2}\right]+\sin(2\theta)+\cos^2(\theta) h\left(\frac{\pi}{2}-\alpha\right)<\frac{\pi}{2}.
		$$
		Since $\frac{\partial G}{\partial \theta}>0$, we have
		$$
		G(\theta,\alpha)\leq G(\arcsin(0.45 \sin \alpha),\alpha)\leq
		\max_{\alpha\in [\frac{\pi}{6},\alpha(2)]}G(\arcsin(0.45 \sin \alpha),\alpha) <\frac{\pi}{2}.
		$$

$(5).$ We can write $\delta$ in the following way:
		$$
		\delta(\Omega)=\frac{2}{\pi}d(\alpha,\theta)+R_0-1,
		$$
		where 
		$d$ has been defined in (\ref{definition_d}). Observe that $d$ is concave with respect to $\theta$. Therefore
		\begin{eqnarray*}
		\min d(\theta,\alpha)&=&
		\min_\alpha\{
		d(\alpha,\arcsin(0.45\sin \alpha)), d(\alpha,\arcsin(\sqrt{2}/2\sin \alpha))
		\}\\
		&\approx& 1.4169.
		\end{eqnarray*}
		On the other hand,
		$$
		\lambda(\Omega)=\frac{4}{\pi}\mathcal{A}_1=\frac{4}{\pi}L(\alpha,\theta),
		$$
		where 
		$$
		L(\alpha,\theta)=g\left(\frac{\pi}{2}-\theta\right)-\frac{\cos^2 \theta}{\cos^2 \alpha}g\left(\frac{\pi}{2}-\theta\right).
		$$
		Now, it is easy to see that $L$ is decreasing with respect to $\theta$. Therefore $L(\alpha,\theta)\leq L(\alpha,\arcsin(0.45 \sin\alpha))$. This implies that
		$$
		\lambda(\Omega)\leq \frac{4}{\pi}\max_{\alpha}L(\alpha,\arcsin(0.45 \sin\alpha))\approx 0.7081.
		$$
		The estimates above on $\delta(\Omega)$ and $\lambda(\Omega)$ entail $\mathcal{F}(\Omega)=\frac{\delta(\Omega)}{\lambda^2(\Omega)}>0.406.$
\end{proof}


\section{Conjecture on the optimal set}\label{section-conjecture}
In this section, we describe  a set that we conjecture to be  optimal for $\mathcal{F}$. 
This conjecture would follow from these  two  properties:
\begin{conj}\label{conj1}
(i) The optimal set $\Omega_0$ is connected and has two perpendicular axes of symmetry.

\noindent (2) The optimal set has exactly two optimal balls $B_1$ and $B_2$ realizing the Fraenkel asymmetry.
\end{conj}
Once these properties (which seem to be difficult) are proved, the problem becomes  finite dimensional. 
More precisely, the optimal set belongs to a class of sets  named \emph{masks} in  \cite{CiLe2}
(see Figure \ref{optimal-domain}). 
Three parameters are sufficient to describe the family of masks $\textsf{M}$ in competition. For that purpose, we will use the $C^1$ regularity of the optimal domain which allows us to consider only $C^1$ competitors. 
Moreover the volume constraint allows us to get rid of one of these parameters, leading to a simple unconstrained
optimization problem in two variables. We point out that the solution of this  minimization problem is a non convex domain $\textsf{M}_0$ such that $\mathcal{F}(\textsf{M}_0)<
\inf_{\Omega\in\mathcal{C}} \mathcal{F}(\Omega)=0.405585\,,
$
where $\mathcal{C}$ be the class of planar convex sets (see Theorem \ref{thmAFN}).

\medskip
We recall that by Proposition \ref{thmqualit}, statement (2), the boundary of the optimal domain is composed
of arcs of circle, the radius of each arc is the same in any connected component of the complementary
of the union of boundaries of the two optimal balls. 
This holds true since  small variations of
the boundary far from the optimal balls do not change these balls (and then the Fraenkel asymmetry).
Therefore, the problem is locally equivalent to minimizing the perimeter with a volume constraint.
Thanks to the symmetry assumption, a mask is thus composed of 8 arcs of circle with three different radii.
\begin{figure}[h!]
\centering
\includegraphics[width=12cm]{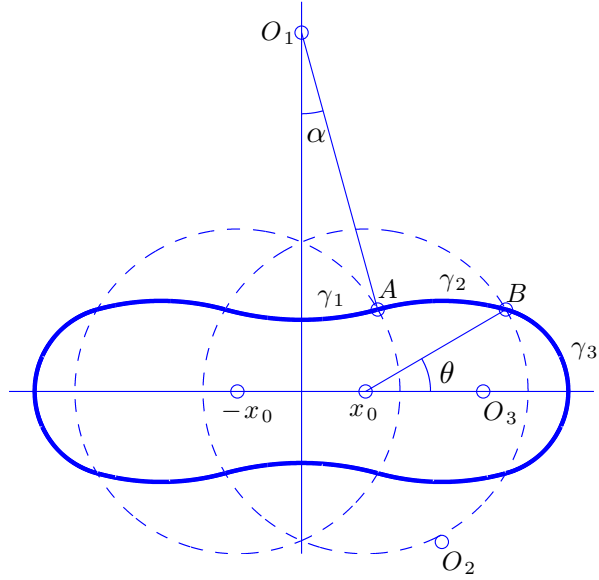}
\caption{Parametrization of a mask $\textsf{M}$ with $\alpha,\theta,x_0$}\label{optimal-domain}
\end{figure}

\medskip
Let us fix the notations (see Figure \ref{optimal-domain}).
We will explain later what are the parameters that we use to completely describe
the sets. We choose to work with sets of area $\pi$ in an orthonormal frame centered at $O$. 
The two optimal (unit) balls $B_1$
and $B_2$ are respectively centered at $P_1=(x_0,0)$ and $P_2=(-x_0,0)$. By symmetry, it suffices to describe the boundary
of $\textsf{M}$ in  $\{x\geq 0, y\geq 0\}$. In this quadrant, the boundary is composed of three
arcs of circle:
\begin{itemize}
\item an arc $\gamma_1$ of center $O_1=(0,y_1)$ and radius $R_1$ in the intersection of the two optimal balls,
\item an arc $\gamma_2$ of center $O_2=(x_2,y_2)$ and radius $R_2$ inside the ball $B_1$ and outside the ball $B_2$,
\item an arc $\gamma_3$ of center $O_3=(x_3,0)$ and radius $R_3$ outside the two balls.
\end{itemize}
We also introduce the intersection points of the boundary of $\textsf{M}$ with the boundary of the optimal balls
in the upper half-plane, $A=(x_A,y_A)$ and $B=(x_B,y_B)$ being in the first quadrant and $A'=(-x_A,y_A),B'=(-x_B,y_B)$
their symmetric with respect to the $y$ axis:
$$\partial\textsf{M}\cap\partial B_1= \{B,A'\} \quad\mbox{and}\quad \partial\textsf{M}\cap\partial B_2= \{A,B'\}.$$
Moreover, by Corollary \ref{cdt_optimalite} these points have same height: $y_A=y_B$.
We finally introduce the angles $\alpha=(\overrightarrow{O_1O},\overrightarrow{O_1A})$ and 
$\theta=(\overrightarrow{P_1O_3},\overrightarrow{P_1B})$.

\medskip
The $C^1$ regularity of $\partial\textsf{M}$  implies that $O_1, A, O_2$ and $O_2,O_3,B$ are on the same line. By elementary trigonometric calculus the following relations hold:
$$
R_1=\frac{\cos\theta -x_0}{\sin\alpha},\quad R_2=\frac{x_0}{\sin\alpha},\quad R_3=\frac{\sin\theta}{\cos\alpha}\,,
$$
$$
O_1=\left(0,\frac{\cos(\theta-\alpha)-x_0\cos(\alpha)}{\sin\alpha}\right),\; O_2=\left(\cos\theta,\sin\theta -
\frac{x_0\cos\alpha}{\sin\alpha}\right),\; O_3=\left(x_0+\frac{\cos(\alpha+\theta)}{\cos(\alpha}\right)\,,
$$
that is, all these quantities can be expressed in term of the three
 parameters $\alpha,\theta, x_0$.
 
With all these formulae in hand, it becomes easy to compute the perimeter $P(\textsf{M})$ of $\textsf{M}$, its area 
$A(\textsf{M})$ and its Fraenkel asymmetry $\lambda(\textsf{M})$. 
More precisely we have:
\begin{equation}\label{perop}
P(\textsf{M})=4\left(\frac{\alpha (x_0+\cos\theta)}{\sin\alpha} - \frac{\alpha\sin\theta}{\cos\alpha} +
\frac{\pi\sin\theta}{2\cos\alpha}\right),
\end{equation}
and, by definition, $\delta(\textsf{M})=P(\textsf{M})/2\pi -1$,
\begin{equation}\label{aireop}
A(\textsf{M})=4\left(\frac{x_0^2-\cos^2\theta}{2}\,h(\alpha)+x_0(\sin\theta+h(\alpha) \cos\theta)+\cos\theta\sin\theta 
+\frac{\sin^2\theta}{2}\,h\left(\frac{\pi}{2}\,-\alpha\right)\right),
\end{equation}
where  $h$ is defined in (\ref{defn_fonctions_gh}), and finally
\begin{equation}\label{lambdop}
\lambda(\textsf{M})=2-\frac{4}{\pi}\left(2x_0 h(\alpha) \cos\theta +\theta +\cos\theta\sin\theta -h(\alpha) \cos^2\theta \right).
\end{equation}
Notice that  the boundary of an optimal domain is composed by arcs of cercles whose radius changes depending on the mutual position of $\partial\textsf{M}$ and the boundary of an optimal ball. That is $\partial\textsf{M}$ changes curvature at the intersection points
of the boundary with the optimal circle, and hence we do not need to check that $B_1$ and $B_2$ are indeed  optimal balls.
More precisely, we are now looking for the best domain, namely the best parameters $\alpha,\theta,x_0$ for the
ratio $\delta/\lambda^2$ with the constraint $A(\textsf{M})=\pi$.
It turns out that the area is a quadratic polynomial in $x_0$, therefore, the constraint $A(\textsf{M})=\pi$
allows us to eliminate the variable $x_0$, by expressing it as a function of $\theta,\alpha$.
By construction, the three parameters must satisfy
$$0\leq \alpha\leq \frac{\pi}{2},\quad -x_0+\cos\theta\geq 0 \quad\alpha+\theta \leq \frac{\pi}{2}$$
the second inequality expresses the fact that the point $A$ must be in the first quadrant and the third
one that the radius $R_3\leq 1$, otherwise the arc $\gamma_3$ would not be outside the ball $B_1$. Thus
the second inequality just means that we will look for the root of the quadratic which is between $0$
and $\cos\theta$.

Finally, the problem reduces to minimize the function of two variables $J(\alpha,\theta):=(P(\textsf{M})-2\pi)/\lambda^2(\textsf{M})$
where $P(\textsf{M})$ and $\lambda(\textsf{M})$ are defined respectively in (\ref{perop}) and (\ref{lambdop}) and
$x_0$ is expressed by $A(\textsf{M})=\pi$ with $A(\textsf{M})$ defined in (\ref{aireop}). 
We just have to assume
$0\leq\alpha, 0\leq\theta, \alpha+\theta\leq \frac{\pi}{2}$. 
{We observe that in the configuration where  the sign of the  curvatures inside the optimal balls is opposite, the parameters satisfy $0\leq \alpha\leq \pi/2$ and $0\leq \theta\leq \pi$.}

A numerical computation provides the explicit values
of the optimal parameters; Figure \ref{optimal-domain} has been drawn by using such values.
For the corresponding  set the value of the functional $\mathcal{F}$ is approximately $0.3931$. This entails that
the optimal set for $\mathcal{F}$ cannot be convex by Theorem \ref{thmAFN}.

\begin{conj}\label{conj2}
The value of the optimal constant is $c^*=2.5436249$ and the set which saturates  is the ``mask'' $\textsf{M}_0$
described above with the following values of the parameters:
$$\alpha=0.2686247,\ \theta=0.5285017,\ x_0=0.3940769.$$
The value of $\mathcal{F}$ for the set $\textsf{M}_0$ is $1/c^*=0.3931397$.
\end{conj}

\section*{Acknowledgements}
We thank B. Kawohl for very useful discussions on the topic of this paper, in particular we thank him for having suggested the idea of Proposition \ref{centre_optimizing_ball}.

This work started while CB was at the Institut Elie Cartan Nancy supported by the ANR CNRS project GAOS (Geometric Analysis of Optimal Shapes), and the research group INRIA CORIDA (Contr\^ole robuste infini-dimensionnel et applications). 
The research of CB is supported by the Fir Project  2013 ``Geometrical and Qualitative Aspects of PDEs''.
The work of GC was partially done during her ``d\'el\'egation CNRS'' at University of Lorraine.
AH is supported by the project ANR-12-BS01-0007-01-OPTIFORM {\it Optimisation de formes} financed by the French Agence Nationale de la Recherche (ANR). 
The three authors have been supported by the Fir Project 2013 ``Geometrical and Qualitative Aspects of PDEs'' in their visitings.

All these institutions are gratefully acknowledged.


\end{document}